\documentclass[final]{siamart190516}
\usepackage{amsfonts}
\usepackage{amsfonts,amsmath,amssymb}
\usepackage{mathrsfs,mathtools,stmaryrd,wasysym}
\usepackage{enumerate}
\usepackage{enumitem}
\usepackage{esint}
\usepackage{graphicx,psfrag}
\usepackage{hyperref}
\usepackage{amsmath,stackengine}
\usepackage{yfonts}
\usepackage{color}
\usepackage{upgreek}

\DeclareMathAlphabet{\mathpzc}{OT1}{pzc}{m}{it}

\newcommand{\T}{\mathscr{T}}

\newsiamremark{remark}{Remark}

\newcommand{\TheTitle}{Fractional, semilinear, and sparse optimal control: a priori error bounds}
\newcommand{\ShortTitle}{Sparse optimal control for semilinear fractional diffusion}
\newcommand{\TheAuthors}{F. Bersetche, F. Fuica, E. Ot\'arola, D. Quero}

\headers{\ShortTitle}{\TheAuthors}

\title{{\TheTitle}\thanks{FB is supported by ANID through FONDECYT grant 3220254. FF is supported by ANID through FONDECYT grant 3230126. EO is partially supported by ANID through FONDECYT grant 1220156. DQ is supported by UTFSM through Programa de Incentivo a la Investigación Científica (PIIC) and by ANID/Subdirecci\'on del Capital Humano/Doctorado Nacional/2021--21210988.}}

\author{Francisco Bersetche\thanks{Departamento de Matem\'atica, Universidad T\'ecnica Federico Santa Mar\'ia, Valpara\'iso, Chile. \email{francisco.bersetche@usm.cl}}
\and
Francisco Fuica\thanks{Facultad de Matem\'aticas, Pontificia Universidad Cat\'olica de Chile, Santiago, Chile. \email{francisco.fuica@mat.uc.cl}}
\and
Enrique Ot\'arola\thanks{Departamento de Matem\'atica, Universidad T\'ecnica Federico Santa Mar\'ia, Valpara\'iso, Chile. \email{enrique.otarola@usm.cl}}
\and
Daniel Quero\thanks{Departamento de Matem\'atica, Universidad T\'ecnica Federico Santa Mar\'ia, Valpara\'iso, Chile \email{daniel.quero@alumnos.usm.cl}}}

\ifpdf
\hypersetup{
  pdftitle={\TheTitle},
  pdfauthor={\TheAuthors}
}
\fi

\date{Draft version of \today.}

\begin{document}

\maketitle

\begin{abstract}
In this work, we use the integral definition of the fractional Laplace operator and study a sparse optimal control problem involving a fractional, semilinear, and elliptic partial differential equation as state equation; control constraints are also considered. We establish the existence of optimal solutions and first and second order optimality conditions. We also analyze regularity properties for optimal variables. We propose and analyze two finite element strategies of discretization: a fully discrete scheme, where the control variable is discretized with piecewise constant functions, and a semidiscrete scheme, where the control variable is not discretized. For both discretization schemes, we analyze convergence properties and a priori error bounds.
\end{abstract}


% REQUIRED

%%%%%%%%%%%%%%%%%%%%%%%%%%%%%%%%%%
%%%%%%%%%%%%%%%%%%%%%%%%%%%%%%%%%%
%%%%%%                     %%%%%%%
%%%%%%  SEC: INTRODUCTION  %%%%%%%
%%%%%%                     %%%%%%%
%%%%%%%%%%%%%%%%%%%%%%%%%%%%%%%%%%
%%%%%%%%%%%%%%%%%%%%%%%%%%%%%%%%%%

\section{Introduction}\label{sec:intro}
Let $d \in \mathbb{N}$ be such that $d \geq 2$. Let $\Omega \subset \mathbb{R}^d$ be an open and bounded domain with Lipschitz boundary $\partial \Omega$. Define the cost functional 
\begin{equation}\label{def:cost_functional}
J(u,q):= \int_{\Omega}L(x,u) \mathrm{d}x 
+ \frac{\lambda}{2}\|q\|_{L^{2}(\Omega)}^{2} + \mu\|q\|_{L^{1}(\Omega)},
\end{equation} 
where $L: \Omega \times \mathbb{R} \rightarrow \mathbb{R}$ denotes an appropriate Caratheodory function, $\lambda > 0$ corresponds to the so-called \emph{regularization} parameter, and $\mu > 0$ denotes a \emph{sparsity} parameter. The necessary assumptions on $L$ are deferred to the section \ref{sec:assump_a_L}. In this paper, we are interested in the analysis and discretization of the following \emph{nonconvex} and \emph{nondifferentiable} optimal control problem for a \emph{fractional}, \emph{semilinear}, and \emph{elliptic} partial differential equation (PDE): Find $\min J(u,q)$ subject to the \emph{state equation}
\begin{equation}\label{def:state_eq}
(-\Delta)^{s}u + a(\cdot,u) = q \text{ in } \Omega,
\qquad
u = 0  \text{ in } \Omega^{c},
\end{equation}
where $\Omega^{c} = \mathbb{R}^{d}\setminus\Omega$ and $s \in (0,1)$, and the \emph{control constraints} 
\begin{equation}\label{def:box_constraints}
q \in \mathbb{Q}_{ad},
\qquad
\mathbb{Q}_{ad}:=\{v \in L^{2}(\Omega):~\alpha \leq v(x) \leq \beta~\mathrm{a.e.}~x\in \Omega \}.
\end{equation}
We adopt the integral definition of the fractional Laplace operator $(-\Delta)^s$. For a smooth function $w: \mathbb{R}^d \rightarrow \mathbb{R}$, the operator $(-\Delta)^s$ is defined as follows:
\begin{equation}\label{eq:pointwise_formula}
(-\Delta)^{s}w(x):= C(d,s)\mathrm{p.v.}\int_{\mathbb{R}^{d}}\frac{w(x) - w(y)}{|x - y|^{d+2s}} \mathrm{d}y,\quad C(d,s):=\frac{2^{2s}s\Gamma(s+\frac{d}{2})}{\pi^{\frac{d}{2}}\Gamma(1-s)},
\end{equation}
where p.v.~stands for the \emph{Cauchy principal value} and $C(d,s)$ is a normalization constant. The necessary assumptions on the function $a$ are deferred until section \ref{sec:assump_a_L}. The control bounds $\alpha$ and $\beta$ are such that $\alpha < 0 < \beta$; see \cite[Remark 2.1]{MR3023751} for a discussion.

The optimal control problem under consideration involves a cost functional $J$ that contains the $L^1(\Omega)$-norm of the control variable. The study of this type of optimal control problems is mainly motivated by the following two observations: First, $\| \cdot \|_{L^{1}(\Omega)}$ is a natural measure of the control cost. Second, $\| \cdot \|_{L^{1}(\Omega)}$ leads to sparsely supported optimal controls, i.e., optimal controls that are non-zero only in a small region of the considered domain. This is a desirable property in applications, for example in the optimal placement of discrete actuators \cite{MR2556849}. From a mathematical point of view, the analysis and discretization of the optimal control problem considered here is anything but trivial and interesting, especially due to the following considerations:

\begin{enumerate}
\item \emph{Fractional diffusion:} The efficient approximation of problems involving the \emph{integral} fractional Laplacian carries two main difficulties. The first and most important is that $(-\Delta)^s$ is a non-local operator \cite{MR3893441,hitchhikers}. The second is the lack of boundary regularity, which leads to reduced convergence rates \cite{MR3893441,MR4283703}.

\item \emph{Non-linearity/Non-convexity:} Since the state equation is a \emph{semilinear} PDE, the control problem is non-convex. Consequently, first order optimality conditions are necessary conditions for local optimality; sufficiency requires the investigation of second order optimality conditions \cite{MR3023751,Troltzsch}.

\item \emph{Non-differentiability:} Due to the presence of $\|.\|_{L^{1}(\Omega)}$ in the cost functional $J$, the optimal control problem becomes non-differentiable ($\alpha < 0 < \beta$). This leads to some difficulties that do not occur in the differential case $\lambda>0$ and $\mu=0$ \cite{Troltzsch}, especially when analysing second order optimality conditions \cite{MR3023751} and investigating finite element techniques \cite{MR3023751,MR2826983}.
\end{enumerate}

For the special case $s = 1$, there are several papers in the literature that provide error estimates for finite element discretizations of control problems related to ours. As far as we know, the first paper is \cite{MR2826983}, in which the authors consider a linear PDE and propose several finite element strategies to discretize the admissible control set. For all strategies considered, the authors obtain bounds for the error that occurs when approximating the optimal control variable in $L^2(\Omega)$. The semilinear scenario was later developed in \cite{MR3023751}. In this paper, the authors develop two strategies for discretization: one based on the variational discretization approach and another where the admissible control set is discretized with piecewise constant functions. Based on a complete study of second order optimality conditions, the authors obtain error estimates in $L^{\infty}(\Omega)$ for the error that occurs when approximating all the optimal variables involved. We would also like to mention \cite{MR2995176} for the piecewise linear approximation of the admissible control set and \cite{MR4122501,MR2826983} for a posteriori error analyses.

For the non-local and linear scenario, $s \in (0,1)$ and $a \equiv 0$, there are several papers in which finite element strategies are analyzed. We refer the interested reader to \cite{MR3990191} for the analysis of a priori error bounds for the differentiable case $\lambda>0$ and $\mu = 0$ and to the more recent work \cite{wang2023adaptive} for the analysis of a posteriori error bounds in the non-differentiable scenario. We also mention \cite{MR4066856,MR3739306}, where the authors consider the linear version of \eqref{def:cost_functional}--\eqref{def:box_constraints}, but with the \emph{spectral} definition of the fractional Laplacian. It is important to mention that the treatment of discretizations of problems involving the spectral and integral definitions of the fractional Laplacian is fundamentally different due to regularity and discretization properties. As far as we know, this work is the \emph{first} to provide a complete analysis for the semilinear control problem \eqref{def:cost_functional}--\eqref{def:box_constraints}, which also includes the development and analysis of finite element strategies.

The structure of this article is as follows. In section \ref{sec:notation_and_prel} we introduce the notation, the functional framework and the assumptions that we will use in our work. In section \ref{sec:state_eq} we give an overview of \eqref{def:state_eq} and its discretization by finite elements. In section \ref{sec:OCP_problem} we present a weak formulation of \eqref{def:cost_functional}--\eqref{def:box_constraints}, analyze existence results, and derive first and second order optimality conditions; furthermore, regularity properties are also analyzed. In section \ref{sec:FD_scheme}, we introduce a fully discrete method and provide convergence properties and error bounds. In section \ref{sec:SD_scheme}, we develop a semidiscretization scheme and derive error bounds. We conclude our work with section \ref{sec:numerical_exp} in which we perform a numerical experiment that illustrates the performance of the proposed methods.

%%%%%%%%%%%%%%%%%%%%%%%%%%%%%%%%%%%%%%%%%%%%%%%%
%%%%%%%%%%%%%%%%%%%%%%%%%%%%%%%%%%%%%%%%%%%%%%%%
%%%%%%                                   %%%%%%%
%%%%%%  SEC: NOTATION AND PRELIMINARIES  %%%%%%%
%%%%%%                                   %%%%%%%
%%%%%%%%%%%%%%%%%%%%%%%%%%%%%%%%%%%%%%%%%%%%%%%%
%%%%%%%%%%%%%%%%%%%%%%%%%%%%%%%%%%%%%%%%%%%%%%%%

\section{Notation and preliminary remarks}
\label{sec:notation_and_prel}

Let us set the notation and recall some facts that will be useful later.

%%%%%%%%%%%%%%%%%%%%%%%%%%%%%%%%%%%
%%%%%%   SUBSEC: NOTATION   %%%%%%%
%%%%%%%%%%%%%%%%%%%%%%%%%%%%%%%%%%%

\subsection{Notation}
We denote by $\Omega^{c}$ the complement of $\Omega$. For normed spaces $\mathscr{X}$ and $\mathscr{Y}$, we write $\mathscr{X}\hookrightarrow \mathscr{Y}$ to denote that $\mathscr{X}$ is continuously embedded in $\mathscr{Y}$. We denote by $\mathscr{X}'$ and $\| \cdot \|_{\mathscr{X}}$ the dual and the norm of $\mathscr{X}$, respectively. The duality pairing between $\mathscr{X}$ and $\mathscr{X}'$ is denoted by $\langle \cdot , \cdot \rangle_{\mathscr{X},\mathscr{X}'}$. When the spaces $\mathscr{X}$ and $\mathscr{X}'$ are clear from the context, we simply write $\langle \cdot , \cdot \rangle$. Let $\{x_{n}\}_{n=1}^{\infty}$ be a sequence in $\mathscr{X}$. We denote by $x_{n} \rightarrow x$, $x_{n} \rightharpoonup x$, and $x_{n} \mathrel{\ensurestackMath{\stackon[1pt]{\rightharpoonup}{\scriptstyle\ast}}} x$ the strong, weak, and weak$^{*}$ convergence, respectively, of $\{x_{n}\}_{n=1}^{\infty}$ to $x$ in $\mathscr{X}$ as $n \uparrow \infty$.  Finally, $\mathfrak{a} \lesssim \mathfrak{b}$ indicates that $\mathfrak{a} \leq C \mathfrak{b}$, where $C$ is a positive constant that does not depend on either $\mathfrak{a}$ or $\mathfrak{b}$. The value of $C$ might change at each occurrence.

%%%%%%%%%%%%%%%%%%%%%%%%%%%%%%%%%%%%%%%%%
%%%%%   SUBSEC: SUBDIFF AND SUBGRAD  %%%%
%%%%%%%%%%%%%%%%%%%%%%%%%%%%%%%%%%%%%%%%%

\subsection{Subgradients and subdifferentials}\label{sec:subdiff_subgrad}

We denote $\mathbb{R} \cup \{ + \infty\}$ by $\mathbb{R}_{\infty}$. Let $j: Z \rightarrow \mathbb{R}_{\infty}$ be a given function, where $Z$ is a real normed space, and let $z$ be a point in $\mathrm{Dom }$ $j$, i.e.,  $z \in Z$ is such that $j(z) < \infty$. An element  $\zeta \in Z'$ is called a \emph{subgradient of $j$ at $z$} if it satisfies the following \emph{subgradient} inequality \cite[Chapter 4.1]{MR3026831}:
\begin{equation}\label{def:subgrad}
j(y) - j(z) \geq \langle \zeta, y - z \rangle_{Z',Z} \qquad \forall y \in Z.
\end{equation}
The set of all subgradients of $j$ at $z$ is denoted by $\partial j(z)$ and called the \emph{subdifferential} of $j$ at $z$. Of particular interest is the case where $Z = L^{1}(\Omega)$ and $j :Z \rightarrow \mathbb{R}_{0}^{+}$ is defined by $j(z) = \| z \|_{L^{1}(\Omega)}$. In this scenario, it follows that $ \zeta \in \partial j(z)$ if and only if
\begin{equation*}
\label{eq:subdiff_L1norm}
\zeta(x) = 1 \text{ if } z(x) > 0,
\quad
\zeta(x) = -1 \text{ if } z(x) < 0,
\quad
\zeta(x) \in [-1,1] \text{ if } z(x) = 0,
\end{equation*}
for a.e.~$x \in \Omega$ \cite[Proposition 4.6.2]{MR2330778}. For $z,v \in L^1(\Omega)$, the directional derivative of $j$ at $z$ in the direction $v$ is given by
\begin{equation}\label{eq:dir_der_g}
j'(z;v) 
= \lim_{\rho \to 0} \frac{j(z +\rho v) - j(z)}{\rho} =
\int_{\Omega^{+}_{z}}v \mathrm{d}x - \int_{\Omega^{-}_{z}}v \mathrm{d}x + \int_{\Omega^{0}_{z}}|v| \mathrm{d}x,
\end{equation} 
where $\Omega^{+}_{z}, \Omega^{-}_{z}$, and $\Omega^{0}_{z}$ denote the sets of points in $\Omega$ where $z$ is positive, negative, and zero, respectively. 
%We also note the relation \cite[Proposition 2.126]{MR1756264}
%\begin{equation}
%j'(z;v) = \max_{\zeta \in \partial j(z)}\int_{\Omega} \zeta v \mathrm{d}x \quad \forall v \in L^1(\Omega).
%\end{equation} %%%%% FF: Esto lo comenté  porque no veo la utilidad de agregarlo por el momento

Let $M \subseteq Z$ be nonempty. The functional $\mathbb{I}_M: Z \rightarrow \mathbb{R}_{\infty}$ defined by $\mathbb{I}_M(z) = 0$ if $z \in M$ and $\mathbb{I}_M(z) = + \infty$ if $z \in Z \setminus M$ is called the \emph{indicator functional of $M$}. We note that $\mathbb{I}_M$ is proper and convex if and only if $M$ is nonempty and convex. Let $z \in M$. It follows from \eqref{def:subgrad} that $\zeta \in \partial \mathbb{I}_M(z)$ if and only if $\langle \zeta, y - z \rangle_{Z',Z} \leq 0$ for all $y \in M$.

%%%%%%%%%%%%%%%%%%%%%%%%%%%%%%%%%%%%%%%%%%
%%%%%%   SUBSEC: FUNCTION SPACES   %%%%%%%
%%%%%%%%%%%%%%%%%%%%%%%%%%%%%%%%%%%%%%%%%%

\subsection{Function spaces}\label{sec:function_spaces}

Let $s \geq 0$ and let $\mathbb{R}^d \ni \xi \mapsto \iota(\xi)= (1 + |\xi|^2)^{\frac{s}{2}} \in \mathbb{R}$. With $\mathcal{F}$ we denote the Fourier transform. We define the fractional Sobolev space $H^{s}(\mathbb{R}^{d}) := \{ v \in L^2(\mathbb{R}^{d}): \iota \mathcal{F}(v) \in  L^2(\mathbb{R}^{d})\}$, which is endowed with the norm $\|v\|_{H^{s}(\mathbb{R}^{d})}:= \|\iota \mathcal{F}(v)\|_{L^2(\mathbb{R}^{d})}$; see \cite[Definition 15.7]{MR2328004} and \cite[Chapter 1, Section 7]{MR0350178}. 

We define $\tilde{H}^{s}(\Omega)$ as the closure of $C_0^{\infty}(\Omega)$ in $H^s(\mathbb{R}^d)$. According to \cite[Theorem 3.29]{MR1742312}, we have the characterization $\tilde{H}^{s}(\Omega)= \{v \in H^{s}(\mathbb{R}^{d}): \text{ supp } v \subset \overline{\Omega} \}$. We endow the fractional Sobolev space $\tilde{H}^{s}(\Omega)$ with the following inner product and norm:
\[
(u,v)_{\tilde{H}^{s}(\Omega)} := \int_{\mathbb{R}^{d}} \int_{\mathbb{R}^{d}} \dfrac{(u(x) - u(y))(v(x) - v(y))}{|x - y|^{d + 2s}}\mathrm{d}x \mathrm{d}y, 
\quad 
\|v\|_{\tilde{H}^{s}(\Omega)} := (v,v)^{\frac{1}{2}}_{\tilde{H}^{s}(\Omega)}.
\]
We denote by $H^{-s}(\Omega)$ the dual space of $\tilde{H}^{s}(\Omega)$. Finally, we introduce
\[
 \mathcal{A}: \tilde{H}^{s}(\Omega) \times \tilde{H}^{s}(\Omega) \rightarrow \mathbb{R}, 
 \quad
 \mathcal{A}(u,v) = 2^{-1}C(d,s)(u,v)_{\tilde{H}^{s}(\Omega)}, 
 \quad
 \|v\|_{s} := \mathcal{A}(v,v)^{\frac{1}{2}}.
\]

We will use the following continuous and compact embedding results repeatedly. Let $s \in (0,1)$. If $\mathfrak{r} \in [1,2d/(d - 2s)]$, then $H^{s}(\Omega)\hookrightarrow L^{\mathfrak{r}}(\Omega)$ \cite[Theorem 7.34]{MR2424078}. If $\mathfrak{r}\in [1,2d/(d - 2s))$, then the embedding is compact \cite[Corollary 7.2]{hitchhikers}.

%%%%%%%%%%%%%%%%%%%%%%%%%%%%%%%%%%%%%%%%%%%%%%%
%%%%%%   SUBSEC: ASSUMPTIONS (ON a and L)  %%%%
%%%%%%%%%%%%%%%%%%%%%%%%%%%%%%%%%%%%%%%%%%%%%%%

\subsection{Assumptions}\label{sec:assump_a_L}

We will operate under the following assumptions on $a$ and $L$. However, we must mention right away that some of the results obtained in this paper are also valid under less restrictive conditions; when possible, we explicitly mention the assumptions on $a$ and $L$ that are required to obtain a particular result.

\begin{enumerate}[label=(A.\arabic*)]
\item \label{A1} $a: \Omega \times \mathbb{R} \rightarrow \mathbb{R}$ is a Carath\'eodory function of class $C^{2}$ with respect to the second variable and $a(\cdot,0) \in L^2(\Omega)\cap L^{r}(\Omega)$ for $r > d/2s$.
\item \label{A2} $\tfrac{\partial a}{\partial u}(x,u) \geq 0$ for a.e.~$x \in \Omega$ and for all $u \in \mathbb{R}$.
\item \label{A3} For all $\mathfrak{m} > 0$, there exists $C_{\mathfrak{m}} > 0$ such that
\[
\sum_{i=1}^{2} \left|\frac{\partial^{i} a}{\partial u^{i}}(x,u)\right| \leq C_{\mathfrak{m}}, 
\qquad
\left|\frac{\partial^{2} a}{\partial u^{2}}(x,u) - \frac{\partial^{2} a}{\partial u^{2}}(x,v) \right| \leq C_{\mathfrak{m}}|u - v|
\]
for a.e.~$x \in \Omega$ and $u,v$ such that $|u|,|v| \leq \mathfrak{m}$.
\end{enumerate}

We note that it follows directly from \ref{A3} and the mean value theorem that $a $ and $\tfrac{\partial a}{\partial u}$ are locally Lipschitz with respect to the second variable.

\begin{enumerate}[label=(B.\arabic*)]
\item \label{B1} $L: \Omega \times \mathbb{R} \rightarrow \mathbb{R}$ is a Carath\'eodory function of class $C^{2}$ with respect to the second variable and $L(\cdot,0) \in L^1(\Omega)$.
\item \label{B2} For all $\mathfrak{m} > 0$, there exist $\psi_{\mathfrak{m}}, \phi_{\mathfrak{m}} \in L^r(\Omega)$ with $r>d/2s$ such that
\[
 \left| \frac{\partial L}{\partial u}(x,u) \right| \leq \psi_{\mathfrak{m}}(x),
\qquad
\left| \frac{\partial^{2} L}{\partial u^{2}}(x,u) \right| \leq \phi_{\mathfrak{m}}(x)
\]
for a.e.~$x \in \Omega$ and $u$ such that $|u| \leq \mathfrak{m}$.
\end{enumerate}

The following assumption is necessary to obtain further regularity properties for optimal control variables and to derive error estimates.
\begin{enumerate}[label=(C.\arabic*)]
\item \label{C2} For all $\mathfrak{m} > 0$ and $u \in [-\mathfrak{m},\mathfrak{m}]$, $\tfrac{\partial L}{\partial u}(\cdot,u) \in L^{2}(\Omega)$ and $\tfrac{\partial^{2} L}{\partial u^{2}}(\cdot,u) \in L^{\frac{d}{s}}(\Omega)$.
\end{enumerate}

%%%%%%%%%%%%%%%%%%%%%%%%%%%%%%%%%%%%
%%%%%%%%%%%%%%%%%%%%%%%%%%%%%%%%%%%%
%%%%%%                       %%%%%%%
%%%%%%  SEC: STATE EQUATION  %%%%%%%
%%%%%%                       %%%%%%%
%%%%%%%%%%%%%%%%%%%%%%%%%%%%%%%%%%%%
%%%%%%%%%%%%%%%%%%%%%%%%%%%%%%%%%%%%

\section{Fractional semilinear PDEs}\label{sec:state_eq}

Let $s \in (0,1)$ and let $q \in L^{r}(\Omega)$ with $r>d/2s$. We introduce the following weak formulation for the \emph{fractional}, \emph{semilinear}, and \emph{elliptic} PDE \eqref{def:state_eq}: Find $u \in \tilde{H}^{s}(\Omega)$ such that
\begin{equation}\label{eq:weak_problem}
\mathcal{A}(u,v) + \int_{\Omega} a(x,u)v\mathrm{d}x = \int_{\Omega} qv \mathrm{d}x \quad \forall v \in \tilde{H}^{s}(\Omega).
\end{equation}
Here, $a = a(x, u) : \Omega \times \mathbb{R} \rightarrow \mathbb{R}$  denotes a Carath\'eodory function which is monotone increasing in $u$. We assume that for every $\mathfrak{m} > 0$ there exists $\varphi_{\mathfrak{m}} \in L^{\mathfrak{t}}(\Omega)$ such that
\begin{equation}
 |a(x,u)| \leq \varphi_{\mathfrak{m}}(x)~\textrm{a.e.}~x \in \Omega, ~u \in [-\mathfrak{m},\mathfrak{m}], \quad \mathfrak{t} = 2d/(d+2s).
 \label{eq:varphi_m}
\end{equation}
If, in addition, $a(\cdot,0) \in L^r(\Omega)$, then \eqref{eq:weak_problem} admits a unique solution $u \in \tilde{H}^{s}(\Omega) \cap L^{\infty}(\Omega)$ satisfying the stability bound $\| u \|_{s} + \| u \|_{L^{\infty}(\Omega)} \lesssim \|q - a(\cdot,0) \|_{L^{r}(\Omega)}$ \cite[Theorem 3.1]{MR4358465}.

\begin{theorem}[Sobolev regularity]\label{thm:sobolev_reg}
Let $s\in(0,1)$ and let $q \in L^{2}(\Omega) \cap L^{r}(\Omega)$. If $a(\cdot,0)\in L^{2}(\Omega)$ and $a$ is locally Lipschitz with respect to the second variable, then the solution $u$ of problem \eqref{eq:weak_problem} belongs to $H^{s + \kappa -\varepsilon}(\Omega)$ for all $0 < \varepsilon  < s$, where $\kappa = \frac{1}{2}$ for $\frac{1}{2} < s < 1$ and $\kappa = s - \varepsilon$ for $0 < s \leq \frac{1}{2}$. Moreover, we have the bound
\begin{equation*}\label{eq:estimate_frac_Lap}
\|u\|_{H^{s + \kappa -\varepsilon}(\Omega)}
\lesssim
C\varepsilon^{-\nu}\|q - a(\cdot,0)\|_{L^2(\Omega)} \qquad \forall \varepsilon \in (0,s),
\end{equation*}
where $\nu = \frac{1}{2}$ for $\frac{1}{2} < s < 1$ and $\nu = \frac{1}{2} + \nu_{0}$ for $0 < s \leq \frac{1}{2}$. Here, $\nu_0$ and $C$ denote positive constants that depend on $\Omega$ and $d$ and $\Omega$, $d$, and $s$, respectively.
\end{theorem}
\begin{proof}
The proof follows from a direct application of \cite[Theorem 2.1 and inequality (2.6)]{MR4283703} using the fact that $a$ is locally Lipschitz with respect to the second variable so that $\| q - a(\cdot,u)\|_{L^2(\Omega)} \lesssim \| q - a(\cdot,0) \|_{L^2(\Omega)} + \| u \|_{L^2(\Omega)} \lesssim \| q - a(\cdot,0) \|_{L^2(\Omega)}$.
\end{proof}

\subsection{Finite element discretization}

We now present a finite element approximation of problem \eqref{eq:weak_problem} under the additional assumption that $\Omega$ is a Lipschitz \emph{polytope}. Let $\{\mathscr{T}_h\}_{h>0}$ be a collection of conforming and quasi-uniform meshes $\T_h$ made of closed simplices $T$, where $h =\max\{ h_T: T \in \mathscr{T}_h \}$ and $h_T = \text{diam}(T)$. For each mesh $\T_h$, we introduce the following standard finite element space:
\begin{equation}\label{def:piecewise_linear_set}
\mathbb{V}_{h}:=\{v_{h}\in C(\bar{\Omega}): v_{h}|_T\in \mathbb{P}_{1}(T) \ \forall T\in \T_{h}, v_h = 0 \text{ on } \partial\Omega\}.
\end{equation}

The discrete approximation of  \eqref{eq:weak_problem} is as follows: Find $\mathsf{u}_{h} \in \mathbb{V}_{h}$ such that
\begin{equation}\label{eq:weak_discrete_problem}
\mathcal{A}(\mathsf{u}_{h},v_{h})
+
\int_{\Omega}a(x,\mathsf{u}_{h})v_h \mathrm{d}x
=
\int_{\Omega} q v_{h} \mathrm{d}x
\quad \forall v_h \in \mathbb{V}_{h}.
\end{equation}
The existence of a discrete solution follows from Brouwer's fixed point theorem; uniqueness follows from the monotonicity of $a$. Moreover, we have $\|\mathsf{u}_{h}\|_{s} \lesssim \|q\|_{H^{-s}(\Omega)}$.

We now state a priori error estimates. For this purpose, we will assume that
\begin{equation}
\label{eq:a_is_Lipschitz_globally}
| a(x,u) - a(x,v) | \leq | \phi(x) | | u -v |
\end{equation}
for a.e.~$x \in \Omega$ and $u,v \in \mathbb{R}$. The function $\phi$ belongs to $L^{\mathfrak{y}}(\Omega)$, where $\mathfrak{y} = d/2s$.
\begin{theorem}[a priori error estimates]
\label{thm:error_estimates_frac_Lap}
Let $s\in(0,1)$ and let $q\in L^{r}(\Omega)$ with $r > d/2s$. Let $u \in \tilde{H}^{s}(\Omega)$ be the solution to \eqref{eq:weak_problem} and let $\mathsf{u}_{h} \in \mathbb{V}_{h}$ be its finite element approximation obtained as the solution to \eqref{eq:weak_discrete_problem}. If $a$ satisfies \eqref{eq:a_is_Lipschitz_globally}, then we have
\begin{equation}\label{eq:quasi_best_approx}
\| u - \mathsf{u}_{h} \|_{s} \lesssim \|u - v_{h}\|_{s} \qquad \forall v_{h} \in \mathbb{V}_{h}.
\end{equation}
If, in addition, $q \in L^{2}(\Omega)$, $a$ is locally Lipschitz with respect to the second variable, and $a(\cdot,0)\in L^{2}(\Omega)$, then we have the error bound
\begin{equation}
 \label{eq:error_in_norm_s}
\|u - \mathsf{u}_{h} \|_{s}
\lesssim
h^{\gamma}|\log h|^{\varphi} \|q - a(\cdot,0)\|_{L^{2}(\Omega)},
\quad
\gamma = \min\{s,\tfrac{1}{2}\},
\end{equation}
If, in addition, $a$ satisfies the condition \eqref{eq:a_is_Lipschitz_globally} with $\mathfrak{y} = d/s$, then we have
\begin{equation}
\label{eq:error_in_norm_L2}
\| u - \mathsf{u}_{h} \|_{L^2(\Omega)}
\lesssim
h^{2\gamma}|\log h|^{2\varphi} \|q - a(\cdot,0)\|_{L^{2}(\Omega)},
\quad
\gamma = \min\{s,\tfrac{1}{2}\}.
\end{equation}
Here, $\varphi = \nu$ if $s\neq \frac{1}{2}$, $\varphi = 1 +\nu$ if $s=\frac{1}{2}$, and $\nu \geq \frac{1}{2}$ is the constant in Theorem \ref{thm:sobolev_reg}.
\end{theorem}
\begin{proof}
The proof of \eqref{eq:quasi_best_approx} can be found in \cite[Theorem 5.2]{MR4358465}. The error estimates \eqref{eq:error_in_norm_s} and \eqref{eq:error_in_norm_L2} can be found in \cite[Theorem 5.1]{MR4599045} and \cite[Theorem 5.2]{MR4599045}.
\end{proof}

We conclude this section with the following convergence result.

\begin{lemma}[convergence]\label{lemma:conv_disc_st_eq}
Let $s\in(0,1)$ and let $u_{h}\in \mathbb{V}_{h}$ be the solution to
\[
\mathcal{A}(u_{h},v_{h}) + \int_{\Omega}a(x,u_{h})v_h \mathrm{d}x = \int_{\Omega} q_{h}v_{h} \mathrm{d}x \quad \forall v_h \in \mathbb{V}_{h},
\]
where $q_{h} \in L^{r}(\Omega)$ with $r>d/2s$. If $a$ satisfies the condition \eqref{eq:a_is_Lipschitz_globally}, then we have $q_{h} \rightharpoonup q \text{ in } L^{r}(\Omega) \Longrightarrow u_{h} \rightarrow u \text{ in } L^{\mathfrak{r}}(\Omega)$ as $h \rightarrow 0$. Here, $\mathfrak{r} \leq 2d/(d - 2s)$.
\end{lemma}
\begin{proof}
See \cite[Proposition 5.3]{MR4358465} for a proof. We note that in the statement of \cite[Proposition 5.3]{MR4358465} it is assumed that $\Omega \in C^2$. However, this assumption does not play any role and the same proof can be performed if $\Omega$ is a Lipschitz polytope.
\end{proof}

%%%%%%%%%%%%%%%%%%%%%%%%%%%%%%%%%%%%%%%%%%%%%
%%%%%%%%%%%%%%%%%%%%%%%%%%%%%%%%%%%%%%%%%%%%%
%%%%%%                                %%%%%%%
%%%%%%  SEC: OPTIMAL CONTROL PROBLEM  %%%%%%%
%%%%%%                                %%%%%%%
%%%%%%%%%%%%%%%%%%%%%%%%%%%%%%%%%%%%%%%%%%%%%
%%%%%%%%%%%%%%%%%%%%%%%%%%%%%%%%%%%%%%%%%%%%%

\section{The optimal control problem}
\label{sec:OCP_problem}

In this section, we present the following weak formulation for the optimal control problem introduced in section \ref{sec:intro}: Find
\begin{equation}\label{eq:weak_min_problem}
\min\{J(u,q): (u,q) \in \tilde{H}^{s}(\Omega) \times \mathbb{Q}_{ad} \}
\end{equation}
subject to the \emph{fractional semilinear}, and \emph{elliptic} state equation
\begin{equation}\label{eq:weak_st_eq}
\mathcal{A}(u,v) + \int_{\Omega} a(x,u)v \mathrm{d}x= \int_{\Omega} qv \mathrm{d}x \qquad \forall v \in \tilde{H}^{s}(\Omega).
\end{equation}
Here, $a = a(x, u) : \Omega \times \mathbb{R} \rightarrow \mathbb{R}$ is a monotonically increasing in $u$ Carath\'eodory function that satisfies \eqref{eq:varphi_m} and $a(\cdot,0) \in L^r(\Omega)$ with $r > d/2s$. As explained in section \ref{sec:state_eq}, problem \eqref{eq:weak_st_eq} is well-posed under these assumptions on $a$. We therefore introduce the \emph{control to state map} $\mathcal{S} :L^r(\Omega )\rightarrow \tilde{H}^s(\Omega) \cap L^{\infty}(\Omega)$ which, given a control $q$, associates to it the unique state $u$ that solves \eqref{eq:weak_st_eq}.

\subsection{Existence of an optimal solution}

The existence of an optimal solution $(\bar u, \bar q) \in \tilde{H}^{s}(\Omega) \times \mathbb{Q}_{ad}$ is as follows.

\begin{theorem}[existence of an optimal solution]
Let $s \in (0,1)$. Let  $L = L(x, u) : \Omega \times \mathbb{R} \rightarrow \mathbb{R}$ be a Carath\'eodory function. Assume that, for every $\mathfrak{m}>0$, there exists $\varphi_{\mathfrak{m}} \in L^r(\Omega)$ with $r>d/2s$ and $\sigma_{\mathfrak{m}} \in L^1(\Omega)$ such that
 \begin{equation}  
|a(x,u)| \leq  \varphi_{\mathfrak{m}}(x),
\quad
|L(x,u)| \leq  \sigma_{\mathfrak{m}}(x)
\quad
\textrm{a.e.}~x \in \Omega,~u \in [-\mathfrak{m},\mathfrak{m}].
\label{eq:assumptions_a_L}
 \end{equation}
Thus, \eqref{eq:weak_min_problem}--\eqref{eq:weak_st_eq} admits at least one solution $(\bar u, \bar q) \in \tilde{H}^{s}(\Omega) \cap L^{\infty}(\Omega) \times \mathbb{Q}_{ad}$. 
\end{theorem}
\begin{proof}
Let $\{ (u_k,q_k) \}_{k \in \mathbb{N}}$ be a minimizing sequence, i.e., for $k \in \mathbb{N}$, $q_k \in \mathbb{Q}_{ad}$ and $u_k = \mathcal{S} q_k \in \tilde{H}^s(\Omega)$ are such that $J(u_k,q_k) \rightarrow \mathsf{j} := \inf \{ J(\mathcal{S}q,q): q \in \mathbb{Q}_{ad} \}$ as $k \uparrow \infty$. The arguments in the proof of \cite[Theorem 4.1]{MR4358465} show that, up to a nonrelabeled subsequence, $q_k \mathrel{\ensurestackMath{\stackon[1pt]{\rightharpoonup}{\scriptstyle\ast}}} \bar{q}$ in $L^{\infty}(\Omega)$ and $u_k \rightharpoonup \bar{u}$ in $\tilde{H}^s(\Omega)$ as $k \uparrow \infty$, where $\bar{u} = \mathcal{S} \bar{q}$. On the other hand, the Lebesgue dominated convergence theorem combined with \eqref{eq:assumptions_a_L} and the fact that $u_k \rightarrow \bar{u}$ in $L^{\mathfrak{r}}(\Omega)$ for every $\mathfrak{r} < 2d/(d-2s)$ show that $| \int_{\Omega}(L(x,u_k(x)) - L(x,\bar{u})) \mathrm{d}x| \rightarrow 0$ as $k \uparrow \infty$. Since $\| \cdot \|_{L^1(\Omega)}$ and the square of $\| \cdot \|_{L^2(\Omega)}$ are continuous and convex in $L^1(\Omega)$ and $L^2(\Omega)$, respectively, we can arrive at $J(\bar{u},\bar{q}) \leq \mathsf{j}$. This completes the proof.
\end{proof}

\subsection{First order necessary optimality conditions}

In this section, we develop necessary first order optimality conditions for \eqref{eq:weak_min_problem}--\eqref{eq:weak_st_eq}. Since this problem is nonconvex, we discuss optimality conditions in the context of local solutions: We say that $\bar{q} \in \mathbb{Q}_{ad}$ is a local solution in $L^{2}(\Omega)$  for \eqref{eq:weak_min_problem}--\eqref{eq:weak_st_eq} if there exists $\varepsilon > 0$ such that
\begin{equation}
  J(\mathcal{S}\bar{q},\bar{q}) \leq J(\mathcal{S}q,q) \quad \forall q \in \mathbb{Q}_{ad}: \|q - \bar{q}\|_{L^{2}(\Omega)} \leq \varepsilon.
  \label{eq:local_solution}
\end{equation}
A solution $\bar q$ is called \emph{strict} if $\bar q$ is the unique solution satisfying \eqref{eq:local_solution} for some $\varepsilon > 0$.

We now introduce $F: L^{2}(\Omega) \rightarrow \mathbb{R}$ and $j: L^{1}(\Omega) \rightarrow \mathbb{R}$ by
\begin{equation}\label{def:F_and_j}
F(q) = \int_{\Omega} L(x,\mathcal{S}q) \mathrm{d}x + \dfrac{\lambda}{2}\|q\|_{L^{2}(\Omega)}^{2}, \qquad j(q) = \|q\|_{L^{1}(\Omega)}.
\end{equation}
We also introduce the \emph{reduced cost functional} $\mathfrak{j}: \mathbb{Q}_{ad} \rightarrow \mathbb{R}$ by $\mathfrak{j}(q) = F(q) + \mu j(q)$.

In what follows, we discuss differentiability properties for $\mathcal{S}$ and $F$.

\begin{proposition}[differentiability properties of $\mathcal{S}$]
Let $s \in (0,1)$ and let $r>d/2s$. Assume that \ref{A1}--\ref{A3} hold. Then, the control to state map  $\mathcal{S}: L^r(\Omega) \rightarrow \tilde{H}^s(\Omega) \cap L^{\infty}(\Omega)$ is of class $C^2$. In addition, if $q,w \in L^r(\Omega)$, then $\phi = \mathcal{S}'(q) w \in \tilde{H}^s(\Omega) \cap L^{\infty}(\Omega)$ corresponds to the unique solution to the problem
 \begin{equation}
  \mathcal{A}(\phi,v) + \int_{\Omega}\frac{\partial a}{\partial u}(x,u)\phi v \mathrm{d}x = \int_{\Omega} wv \mathrm{d}x \quad \forall v \in \tilde{H}^s(\Omega),
  \label{eq:phi}
 \end{equation}
where $u = \mathcal{S} q$. If $w_1, w_2 \in L^r(\Omega)$, then $\psi = \mathcal{S}''(q)(w_1,w_2) \in \tilde{H}^s(\Omega) \cap L^{\infty}(\Omega)$ corresponds to the unique solution to
 \begin{equation}
  \mathcal{A}(\psi,v) + \int_{\Omega}\frac{\partial a}{\partial u}(x,u)\psi v \mathrm{d}x = -  \int_{\Omega} \frac{\partial^2 a}{\partial u^2}(x,u)\phi_{w_1}\phi_{w_2} v \mathrm{d}x \quad \forall v \in \tilde{H}^s(\Omega),
  \label{eq:psi}
\end{equation}
where $u = \mathcal{S} q$ and $\phi_{w_i} = \mathcal{S}'(q) w_i$, with $i \in \{1,2\}$.
\label{pro:diff_S}
\end{proposition}
\begin{proof}
See \cite[Theorem 4.3]{MR4358465} for a proof.
\end{proof}

To present the following result, we introduce the \emph{adjoint state} $p$, which corresponds to the solution to the problem: Find $p \in \tilde{H}^{s}(\Omega) \cap L^{\infty}(\Omega)$ such that
\begin{equation}\label{eq:adjoint_eq}
\mathcal{A}(v,p) + \int_{\Omega}\dfrac{\partial a}{\partial u}(x,u)pv \mathrm{d}x = \int_{\Omega} \dfrac{\partial L}{\partial u}(x,u)v \mathrm{d}x \quad \forall v \in \tilde{H}^{s}(\Omega),
\end{equation}
where $u = \mathcal{S}q \in \tilde{H}^{s}(\Omega) \cap L^{\infty}(\Omega) $. Since $\partial a/ \partial u (x,u) \geq 0$ for a.e.~$x\in \Omega$ and for all $u \in \mathbb{R}$ and $\partial L / \partial u (\cdot,u) \in L^r(\Omega)$ for every $\mathfrak{m}>0$ and $u \in [-\mathfrak{m},\mathfrak{m}]$, the adjoint problem \eqref{eq:adjoint_eq} is well-posed.

\begin{proposition}[differentiability properties of $F$]
Let $s \in (0,1)$ and let $r>d/2s$. Assume that \ref{A1}--\ref{A3} and \ref{B1}--\ref{B2} hold. Then, $F: L^2(\Omega) \cap L^r(\Omega) \rightarrow \mathbb{R}$ is of class $C^2$. In addition, if $q, w \in L^2(\Omega) \cap L^r(\Omega)$, then
\begin{equation}
  F'(q)w = \int_{\Omega} \left( p + \lambda q \right) w \mathrm{d}x,
  \label{eq:charac_1_der}
\end{equation}
where $p$ solves \eqref{eq:adjoint_eq}. If $w_1, w_2 \in L^2(\Omega) \cap L^r(\Omega)$, then we have
\begin{equation}
 F''(q)(w_1,w_2) = \int_{\Omega} \left( \frac{\partial^2 L}{\partial u^2} (x,u) \phi_{w_1}\phi_{w_2} + \lambda w_1 w_2 - p  \frac{\partial^2 a}{\partial u^2} (x,u)\phi_{w_1}\phi_{w_2} \right) \mathrm{d}x,
  \label{eq:charac_2_der}
\end{equation}
where $u = \mathcal{S} q$ and $\phi_{w_i} = \mathcal{S}'(q) w_i$, with $i \in \{1,2\}$.
\label{pro:diff_F}
\end{proposition}
\begin{proof}
See \cite[Proposition 4.5]{MR4358465} for a proof.
\end{proof}

The necessary first order optimality conditions are as follows.

\begin{theorem}[first order optimality conditions]\label{thm:first_opt_cond}
If $\bar{q} \in \mathbb{Q}_{ad}$ is a local solution to \eqref{eq:weak_min_problem}--\eqref{eq:weak_st_eq}, then there exists $\bar{\eta} \in \partial j(\bar{q})$ such that
\begin{equation}\label{eq:var_ineq}
\int_{\Omega}(\bar{p} + \lambda \bar{q} + \mu \bar{\eta})(q - \bar{q})\mathrm{d}x \geq 0 \quad \forall q \in \mathbb{Q}_{ad}.
\end{equation}
Here, $\bar{p}$ denotes the solution to \eqref{eq:adjoint_eq} where $u$ is replaced by $\bar{u} = \mathcal{S}\bar{q}$.
\end{theorem}
\begin{proof}
Since $\bar{q}$ is a local solution, there exists $\varepsilon >0$ such that $\mathfrak{j}(\bar{q}) \leq \mathfrak{j}(q)$ for every $q \in \mathbb{Q}_{ad}$ such that $\| \bar{q} - q \|_{L^2(\Omega)} \leq \varepsilon$. Let $q \in \mathbb{Q}_{ad}$ and let $\rho \in (0,1)$ be sufficiently small so that $q_{\rho} := \bar{q} + \rho(q - \bar{q}) = \rho q + (1-\rho) \bar{q} \in \mathbb{Q}_{ad}$ satisfies $\| \bar{q} - q_{\rho} \|_{L^2(\Omega)} \leq \varepsilon$. Thus
\begin{equation}
 0 \leq \mathfrak{j}(q_{\rho}) - \mathfrak{j}(\bar{q}) = [F(q_{\rho}) - F(\bar{q})] + \mu[j(q_{\rho}) - j(\bar{q})].
 \label{eq:local_optimality_step_1}
\end{equation}
Note that $j(q_{\rho}) - j(\bar{q}) \leq \rho( j(q) - j(\bar{q}) )$ because $j$ is convex. Dividing \eqref{eq:local_optimality_step_1} by $\rho$ and taking the limit as $\rho \downarrow 0$ yield $0 \leq F'(\bar{q})(q-\bar{q}) + \mu( j(q) - j(\bar{q}) )$ for every $q \in \mathbb{Q}_{ad}$. This inequality can be rewritten as follows: $\bar{q} \in \mathbb{Q}_{ad}$ is such that
\[
\ell(\bar q) \leq \ell(q) \quad
\forall q \in \mathbb{Q}_{ad},
\quad
\ell: L^2(\Omega) \rightarrow \mathbb{R},
\quad
\ell(q):= F'(\bar{q})q + \mu j(q),
% + \mathbb{I}_{\mathbb{Q}_{ad}}(q).
\]
i.e., $\bar{q}$ is a global minimizer of $\ell$ over $\mathbb{Q}_{ad}$. We must thus have that $0 \in \partial ( \ell + \mathbb{I}_{\mathbb{Q}_{ad}})(\bar{q}) = \partial \ell(\bar{q}) + \partial  \mathbb{I}_{\mathbb{Q}_{ad}}(\bar{q})$ \cite[page 134 and Proposition 4.5.1]{MR2330778}. Here, $\mathbb{I}_{\mathbb{Q}_{ad}}$ corresponds to the indicator functional of $\mathbb{Q}_{ad}$; see section \ref{sec:subdiff_subgrad}. We now use that $F'(\bar q) = (\bar p + \lambda \bar q)$ and let $\bar{\eta} \in \partial j(\bar q)$ to arrive at $-( \bar p + \lambda \bar q) - \mu \bar{\eta} \in \partial  \mathbb{I}_{\mathbb{Q}_{ad}}(\bar{q})$. This allows us to conclude.
% \textsc{i)} follows from slight modifications of the arguments presented in the proof of \cite[Lemma 2.26]{Troltzsch}. Finally, a proof for \textsc{ii)} and \textsc{iii)} can be found in \cite[Corollary 3.2]{MR3023751}. For brevity, we skip details.
%abstract results like the ones presented in \cite[Sections 2.4 and 6.3]{MR1756264}.
\end{proof}

Let $\mathfrak{a},\mathfrak{b} \in \mathbb{R}$ be such that $\mathfrak{a} < \mathfrak{b}$. We introduce the operator $\Pi_{[\mathfrak{a},\mathfrak{b}]}: L^1(\Omega) \rightarrow \mathbb{Q}_{ad}$ by $\Pi_{[\mathfrak{a},\mathfrak{b}]}(v) = \min\{\mathfrak{b},\max\{\mathfrak{a},v\}\}$ and present the following result.

\begin{theorem}[projection formulas]
\label{thm:projection_formulas}
If $\bar{q}, \bar{u}, \bar{p}$, and $\bar{\eta}$ are as in the statement of Theorem \ref{thm:first_opt_cond}, then
\begin{align}
\label{eq:charac_q}
& \bar{q}(x) = \Pi_{[\alpha ,\beta]}\left(-\lambda^{-1}\left(\bar{p}(x) + \mu \bar{\eta}(x)\right) \right),
\quad
\bar{q}(x) = 0 \Leftrightarrow |\bar{p}(x)| \leq \mu, \\
\label{eq:charac_eta}
& \bar{\eta}(x) = \Pi_{[-1,1]}\left(-\mu^{-1}\bar{p}(x) \right)
\end{align}
for a.e.~$x \in \Omega$. In particular, the subgradient $\bar{\eta} \in \partial j(\bar{q})$ is uniquely determined and both $\bar{q}$ and $\bar{\eta}$ belong to $\tilde{H}^{s}(\Omega) \cap L^{\infty}(\Omega)$.
\end{theorem}
\begin{proof}
The derivation of the projection formula for $\bar{q}$ in \eqref{eq:charac_q} is standard in PDE-constrained optimization. The equivalence $\bar{q}(x) = 0 \Leftrightarrow |\bar{p}(x)| \leq \mu$ for a.e.~$x \in \Omega$ can be found in \cite[Corollary 3.2]{MR3023751}. The projection formula for $\bar{\eta}$ in \eqref{eq:charac_eta} can also be found in \cite[Corollary 3.2]{MR3023751}. This formula directly guarantees the uniqueness of $\bar{\eta}$. The desired regularity properties for $\bar{q}$ and $\bar{\eta}$ follow from the projection formulas in \eqref{eq:charac_q} and \eqref{eq:charac_eta}, the fact that $\max\{0,\tau\} = (\tau + |\tau|)/2$ for all $\tau \in \mathbb{R}$, and \cite[Theorem 1]{MR1173747}.
\end{proof}

\subsection{Second order optimality conditions}

Let $\bar{q} \in \mathbb{Q}_{ad}$ be a local minimum and let $\bar{\eta} \in \partial j(\bar{q})$ be the corresponding subgradient. We define
\begin{equation}\label{eq:cone_Cq}
C_{\bar{q}} := \{w \in L^{2}(\Omega): w \text{ satisfies } \eqref{eq:sign_cond} \text{ and } F'(\bar{q})w + \mu j'(\bar{q};w) = 0\},
\end{equation}
where $F'(\bar{q})(\cdot)$ and $j'(\bar{q},\cdot)$ are described in Proposition \ref{pro:diff_F} and \eqref{eq:dir_der_g}, respectively, and
\begin{equation}\label{eq:sign_cond}
w(x) \geq 0 \text{ if } \bar{q}(x) = \alpha,
\qquad
w(x) \leq 0 \text{ if } \bar{q}(x) = \beta.
\end{equation}
The set $C_{\bar{q}}$ is a closed and convex cone in $L^{2}(\Omega)$ \cite[Proposition 3.4]{MR3023751}.

We formulate necessary second order optimality conditions as follows.

\begin{theorem}[second order necessary optimality conditions]
If $\bar{q} \in \mathbb{Q}_{ad}$ is a local minimum of problem \eqref{eq:weak_min_problem}--\eqref{eq:weak_st_eq}, then $F''(\bar{q})w^{2} \geq 0$ for every $w \in C_{\bar{q}}$.
\end{theorem}
\begin{proof}
The proof follows the same arguments as in \cite[Theorem 3.7]{MR3023751}.
\end{proof}

To present the following result, for $\tau>0$,  we introduce the cone
\begin{equation}\label{eq:cone_Cq_tau}
C_{\bar{q}}^{\tau} := \left \{w \in L^{2}(\Omega): w \text{ satisfies } \eqref{eq:sign_cond} \text{ and } F'(\bar{q})w + \mu j'(\bar{q};w) \leq \tau\|w\|_{L^{2}(\Omega)} \right \}.
\end{equation}

\begin{theorem}[equivalence]\label{thm:equivalence}
Let $\bar{q} \in \mathbb{Q}_{ad}$ be a local minimum and let $\bar{\eta} \in \partial j(\bar{q})$ be the corresponding subgradient such that \eqref{eq:var_ineq} holds. Then, the following statements are equivalent:
\begin{itemize}
\item[(i)]
$F''(\bar{q})w^{2} > 0$ for all $w \in C_{\bar{q}}\setminus \{0\}$.
\item[(ii)] There exist $\tau, \delta > 0$ such that $F''(\bar{q})w^{2} \geq \delta\|w\|^{2}_{L^{2}(\Omega)}$ for all $w \in C^{\tau}_{\bar{q}}$.
\end{itemize}
\end{theorem}
\begin{proof}
Since $C_{\bar{q}} \subset C_{\bar{q}}^{\tau}$ for every $\tau > 0$, it is immediate that \emph{(ii)} implies \emph{(i)}.

We now prove that \emph{(i)} implies \emph{(ii)}. To do so, we follow the arguments of the proof of \cite[Theorem 3.8]{MR3023751} and proceed by contradiction. Indeed, we assume the existence of a sequence $\{ v_k\}_{k \in \mathbb{N}}$ such that
\[
 v_{k} \in C_{\bar{q}}^{1/k},
 \quad
 F''(\bar{q})v_{k}^{2} < k^{-1}\|v_{k}\|^2_{L^{2}(\Omega)},
 \quad
 k \in \mathbb{N}.
\]
Define $w_{k}:= \|v_{k}\|_{L^{2}(\Omega)}^{-1}v_{k}$. Note that $w_{k} \in C_{\bar{q}}^{1/k}$ because $C_{\bar{q}}^{1/k}$ is a cone. Moreover,
\begin{equation}\label{eq:step0_equivalence}
\|w_{k}\|_{L^{2}(\Omega)} = 1,
\quad
\ F''(\bar{q})w_{k}^{2} < k^{-1},
\quad
k \in \mathbb{N}.
\end{equation}
Since $\{ w_k\}_{k \in \mathbb{N}}$ is uniformly bounded in $L^2(\Omega)$, we can extract a nonrelabeled subsequence such that $w_k \rightharpoonup w$ in $L^2(\Omega)$ as $k \uparrow \infty$. Note that $w$ satisfies \eqref{eq:sign_cond}. Apply \cite[Lemma 3.5]{MR3023751} to derive that $F'(\bar{q})w + \mu j'(\bar{q};w) \geq 0$. On the other hand, since $j'(\bar{q}; \cdot )$ is weakly lower semicontinuous in $L^2(\Omega)$ and $F'(\bar{q})w_k \rightarrow F'(\bar{q})w$ as $k \uparrow \infty$ we obtain
\[
 F'(\bar{q})w + \mu j'(\bar{q};w) \leq \liminf_{k \uparrow \infty} F'(\bar{q})w_k + \mu j'(\bar{q};w_k) \leq \liminf_{k \uparrow \infty} k^{-1} = 0.
\]
To obtain the last inequality we used that $w_k \in C_{\bar{q}}^{1/k}$. This proves that $F'(\bar{q})w + \mu j'(\bar{q};w) = 0$ and thus that $w \in C_{\bar{q}}$.

We now prove that $w \equiv 0$. Since \emph{(i)} holds, it suffices to prove that $F''(\bar{q})w^{2} \leq 0$. To accomplish this task, we use the characterization \eqref{eq:charac_2_der} and write
\begin{equation}\label{eq:step1_equivalence}
F''(\bar{q})w_{k}^{2} = \int_{\Omega}\left( \dfrac{\partial^{2} L}{\partial u^{2}}(x, \bar{u})\phi_{k}^{2} +
\lambda w_{k}^{2}
- \bar{p} \dfrac{\partial^{2} a}{\partial u^{2}}(x, \bar{u})\phi_{k}^{2} \right) \mathrm{d}x.
\end{equation}
Here, $\phi_{k} = \mathcal{S}'(\bar{q})w_{k} \in \tilde{H}^{s}(\Omega) \cap L^{\infty}(\Omega)$ solves \eqref{eq:phi} with $u$ and $w$ replaced by $\bar{u} = \mathcal{S}(\bar{q})$ and $w_k$, respectively and $\bar{p} \in \tilde{H}^{s}(\Omega) \cap L^{\infty}(\Omega)$ solves \eqref{eq:adjoint_eq} with $u$ replaced by $\bar{u}$. Since $w_k \rightharpoonup w$ in $L^2(\Omega)$ as $k \uparrow \infty$, we deduce that  $\phi_{k} \rightharpoonup \phi = \mathcal{S}'(\bar{q})w$ in $\tilde{H}^{s}(\Omega)$ as $k \uparrow \infty$. Note that $\phi $ solves \eqref{eq:phi} with $u$ replaced by $\bar{u}$. This convergence result and the compact embedding of section \ref{sec:function_spaces} show that $\phi_{k} \rightarrow \phi$ in $L^{\mathfrak{r}}(\Omega)$ for every $\mathfrak{r} < 2d/(d - 2s)$. Thus, we invoke the assumptions \ref{A3} and \ref{B2} and obtain
\begin{align*}
\left| \int_{\Omega}\dfrac{\partial^{2} L}{\partial u^{2}}(x,\bar{u})\left(\phi_{k}^{2} - \phi^{2} \right) \mathrm{d}x \right|
&
\leq
\left\| \dfrac{\partial^{2} L}{\partial u^{2}}(\cdot,\bar{u}) \right\|_{L^{r}(\Omega)}
\|\phi_{k} - \phi\|_{L^{t}(\Omega)}
\|\phi_{k} + \phi\|_{L^{t}(\Omega)},
\\
\left| \int_{\Omega}\dfrac{\partial^{2} a}{\partial u^{2}}(x,\bar{u})\bar{p}\left(\phi_{k}^{2} - \phi^{2} \right) \mathrm{d}x \right|
&
\leq
\left\| \dfrac{\partial^{2} a}{\partial u^{2}}(\cdot,\bar{u}) \right \|_{L^{r}(\Omega)}
\| \bar{p} \|_{L^{\infty}(\Omega)}
\|\phi_{k} - \phi\|_{L^{t}(\Omega)}
\|\phi_{k} + \phi \|_{L^{t}(\Omega)}.
\end{align*}
Here $r^{-1} + 2t^{-1} = 1$. Note that, since $r>d/2s$, we have that $t < 2d/(d-2s)$.

Since the square of $\| \cdot \|_{L^2(\Omega)}$ is continuous and convex in $L^2(\Omega)$, the aforementioned convergence results based on  \eqref{eq:step1_equivalence} and \eqref{eq:step0_equivalence} yield
\[
F''(\bar{q})w^{2} \leq \liminf_{k \rightarrow \infty} F''(\bar{q})w_{k}^{2} \leq 0.
\]
The condition \emph{(i)}, which reads: $F''(\bar{q})w^{2} > 0$ for all $w \in C_{\bar{q}}\setminus \{0\}$, therefore implies that $w = 0$ and that $F''(\bar{q})w_{k}^{2} \rightarrow 0$ as $k \uparrow \infty$.

After proving that $w = 0$, we conclude the proof by arriving at a contradiction. We begin by noting that $w_{k} \rightharpoonup 0$ in $L^{2}(\Omega)$ implies that $\phi_{k} \rightarrow 0$ in $L^{\mathfrak{r}}(\Omega)$ for every $\mathfrak{r} < 2d/(d - 2s)$. We use the latter convergence result and $\|w_{k}\|_{L^{2}(\Omega)} = 1$ to obtain
\[
F''(\bar{q})w_{k}^{2}
=
\lambda
+
\int_{\Omega}\left( \dfrac{\partial^{2} L}{\partial u^{2}}(x, \bar{u})\phi_{k}^{2} - \bar{p}\dfrac{\partial^{2} a}{\partial u^{2}}(x, \bar{u})\phi_{k}^{2}\right) \mathrm{d}x \rightarrow \lambda > 0, \qquad k \uparrow \infty.
\] 
This is a contradiction because $F''(\bar{q})w_{k}^{2} \rightarrow 0$ as $k \uparrow \infty$. This concludes the proof.
\end{proof}

We conclude this section with second order sufficient optimality conditions.

\begin{theorem}[sufficient second order optimality conditions]
\label{thm:second_order_conditions}
Let $\bar{q} \in \mathbb{Q}_{ad}$ be a local minimum  and let $\bar{\eta} \in  \partial j(\bar{q}) $ be the corresponding subgradient such that \eqref{eq:var_ineq} holds. If $F''(\bar{q})w^{2} > 0$ for all $w \in C_{\bar{q}} \setminus \{0\}$, then there exist $\delta > 0$ and $\varepsilon > 0$ such that
 \[
  \mathfrak{j}(q) \geq \mathfrak{j}(\bar{q}) + \tfrac{\delta}{4}\|q - \bar{q}\|_{L^{2}(\Omega)}^{2} \qquad \forall q \in \mathbb{Q}_{ad}: \|q - \bar{q}\|_{L^{2}(\Omega)} \leq \varepsilon.
 \]
\end{theorem}
\begin{proof}
The proof follows the same arguments as in \cite[´Theorem 3.9]{MR3023751}.
\end{proof}

\subsection{Regularity properties}\label{sec:regularity_properties}

Let $\bar{q}$ be a local minimum of problem \eqref{eq:weak_min_problem}--\eqref{eq:weak_st_eq}. In this section, we provide regularity results for all involved optimal control variables, i.e., $\bar{q},\bar{u},\bar{p}$, and $\bar{\eta}$. To accomplish this task, we assume that the conditions \ref{A1}--\ref{A3} and \ref{B1}--\ref{B2} hold together with the additional property \ref{C2}.

To present the following result, we define
\begin{equation}\label{eq:Lambda_L_u}
 \Lambda(L,u) = \left\|\dfrac{\partial L}{\partial u}(\cdot,u)\right\|_{L^{2}(\Omega)}.
\end{equation}

\begin{theorem}[regularity properties]\label{thm:reg_prop_u_p}
Let $\bar{q}$ be a local minimum of \eqref{eq:weak_min_problem}--\eqref{eq:weak_st_eq} and let $\bar{u},\bar{p}$, and $\bar{\eta}$ be the associated optimal variables. Then, $\bar{u},\bar{p},\bar{q},\bar{\eta} \in H^{s+\kappa - \varepsilon}(\Omega)$ for all $\varepsilon \in (0,s)$, where $\kappa = \frac{1}{2}$ for $\frac{1}{2} < s < 1$ and $\kappa = s - \varepsilon$ for $0 < s \leq \frac{1}{2}$. Moreover,
\begin{align}
\label{eq:reg_est_u}
\|\bar{u}\|_{H^{s+\kappa - \varepsilon}(\Omega)} & \lesssim \varepsilon^{-\nu}\|\bar{q} - a(\cdot,0)\|_{L^{2}(\Omega)},
\\
\label{eq:reg_est_p}
\|\bar{p}\|_{H^{s+\kappa - \varepsilon}(\Omega)} & \lesssim \varepsilon^{-\nu}\Lambda(L,\bar{u}),
\end{align}
and
\begin{align}
\label{eq:reg_est_q_eta}
\|\bar{q}\|_{H^{s+\kappa - \varepsilon}(\Omega)}
+
\|\bar{\eta}\|_{H^{s+\kappa - \varepsilon}(\Omega)}
\lesssim \varepsilon^{-\nu}
\left[
1 + \Lambda(L,\bar{u})
\right],
\end{align}
where $\nu = \frac{1}{2}$ for $\frac{1}{2} < s < 1$ and $\nu = \frac{1}{2} + \nu_{0}$ for $0 < s \leq \frac{1}{2}$, with $\nu_{0} = \nu_{0}(\Omega,d) > 0$. The hidden constant in all inequalities is independent of $\varepsilon$.
\end{theorem}
\begin{proof}
Since $\bar{q} \in \mathbb{Q}_{ad} \subset L^{2}(\Omega)\cap L^{r}(\Omega)$ and $a$ is locally Lipschitz with respect to the second variable and satisfies \ref{A1}, the desired regularity property for $\bar{u}$ follows from Theorem \ref{thm:sobolev_reg}. We recall that $\bar{u} \in L^{\infty}(\Omega)$ and that $\| \bar u \|_{L^{\infty}(\Omega)} \lesssim \| \bar q - a(\cdot,0)\|_{L^{r}(\Omega)}$. We now focus on the optimal adjoint variable $\bar{p}$. We first note that since $\bar{u} \in L^{\infty}(\Omega)$ and, for every $\mathfrak{m} >0$, $|\partial L/\partial u(x,u)| \leq \psi_{\mathfrak{m}}(x)$ for a.e.~$x \in \Omega$ and $u \in [-\mathfrak{m},\mathfrak{m}]$, with $\psi_{\mathfrak{m}} \in L^{r}(\Omega)$, we can conclude that $\bar{p} \in L^{\infty}(\Omega)$ \cite[Theorem 3.1]{MR4358465}. With this regularity result, we invoke the assumptions \ref{A3} and \ref{C2} to arrive at
\[
 \frac{\partial L}{\partial u}(\cdot,\bar{u}) - \frac{\partial a}{\partial u}(\cdot,\bar{u})\bar{p}\in L^{2}(\Omega).
\]
Thus, an application of \cite[Theorem 2.1 and inequality (2.6)]{MR4283703} show that $\bar{p} \in H^{s+\kappa - \varepsilon}(\Omega)$ together with the estimate
\begin{equation*}
\|\bar{p}\|_{H^{s+\kappa - \varepsilon}(\Omega)}
\lesssim
\varepsilon^{-\nu}
\left(
\Lambda(\cdot, \bar{u})
% \left\| \dfrac{\partial L}{\partial u}(\cdot,\bar{u})\right\|_{L^{2}(\Omega)}
+ \|\bar{p}\|_{L^2(\Omega)} \right) \lesssim
\varepsilon^{-\nu}
\Lambda(\cdot, \bar{u}).
\end{equation*}
We note that since $\bar{u} \in L^{\infty}(\Omega)$ and the assumption \ref{A3} holds, $| \partial a/\partial u (x,\bar{u})| \leq C_{\mathfrak{m}}$ for a.e.~$x \in \Omega$. The desired regularity property for $\bar \eta$ and part of the estimate in \eqref{eq:reg_est_q_eta} follow from the projection formula \eqref{eq:charac_eta}, the fact that $\max\{0,\tau\} = (\tau + |\tau|)/2$ for all $\tau \in \mathbb{R}$, and \cite[Theorem 1]{MR1173747}, which applies because $s + \kappa -\varepsilon = s + \frac{1}{2} -\varepsilon < \tfrac{3}{2}$ if $\frac{1}{2} < s < 1$ and $s + \kappa -\varepsilon = 2s -2\varepsilon \leq 1 -2\varepsilon < \tfrac{3}{2}$ if $0 < s \leq \frac{1}{2}$. The desired regularity property for $\bar{q}$ follows the same arguments. This concludes the proof.
\end{proof}

\section{A fully discrete scheme for the optimal control problem}\label{sec:FD_scheme}

Before we begin our analysis, we would like to mention that in this section we assume that $\Omega$ is a Lipschitz \emph{polytope}. In the following, we propose and analyze the following fully discrete finite element approximation of our optimal control problem \eqref{eq:weak_min_problem}--\eqref{eq:weak_st_eq}: Find $\min J(u_{h},q_{h})$ subject to the \emph{discrete state equation}
\begin{equation}
\label{eq:discrete_st_eq_FD}
\mathcal{A}(u_{h},v_{h})
+
\int_{\Omega}a(x,u_{h})v_{h} \mathrm{d}x
=
\int_{\Omega}q_{h}v_{h} \mathrm{d}x
\quad
\forall v_h \in \mathbb{V}_h
\end{equation}
and the \emph{control constraints} $q_{h} \in \mathbb{Q}_{ad,h}$. Here, $\mathbb{Q}_{ad,h}:= \mathbb{Q}_{ad} \cap \mathbb{Q}_{h}$, where $\mathbb{Q}_{h} := \{ q_h \in L^\infty(\Omega): q_{h}|_T\in \mathbb{P}_0(T) \ \forall T\in \T_{h}\} $. $\mathbb{V}_{h}$ denotes the finite element space defined in \eqref{def:piecewise_linear_set}.

The existence of at least one optimal solution is standard. To provide first order optimality conditions, we introduce
$
F_{h}: L^2(\Omega) \rightarrow \mathbb{R}
$
and
$
j_{h}: \mathbb{Q}_h \rightarrow \mathbb{R}
$
to be such that
\begin{equation*}
F_h(q) = \int_{\Omega} L(x,\mathcal{S}_h q) \mathrm{d}x + \dfrac{\lambda}{2}\| q \|_{L^{2}(\Omega)}^{2},
\qquad j_h(q_h) = \|q_h\|_{L^{1}(\Omega)},
\end{equation*}
where $\mathcal{S}_{h}: \mathbb{Q}_{ad,h} \rightarrow \mathbb{V}_{h}$ denotes the \emph{discrete control to state map}.

First order optimality conditions for the fully discrete scheme are as follows.

\begin{theorem}[first order optimality conditions]\label{thm:first_opt_cond_discrete}
If $\bar{q}_h \in \mathbb{Q}_{ad,h}$ is a local solution to the fully discrete scheme, then there exists $\bar{\eta}_h \in \partial j_h(\bar{q}_h)$ such that
\begin{equation}\label{eq:var_ineq_discrete}
\int_{\Omega}(\bar{p}_h + \lambda \bar{q}_h + \mu \bar{\eta}_h)(q_h - \bar{q}_h)\mathrm{d}x \geq 0 \quad \forall q_h \in \mathbb{Q}_{ad,h}.
\end{equation}
Here, $\bar{p}_h$ is the solution to the discrete adjoint equation: Find $\bar{p}_h \in \mathbb{V}_h$ such that
\begin{equation}\label{eq:adjoint_eq_discrete}
\mathcal{A}(v_h,\bar{p}_h) + \int_{\Omega}\dfrac{\partial a}{\partial u}(x,\bar{u}_h)\bar{p}_hv_h \mathrm{d}x = \int_{\Omega} \dfrac{\partial L}{\partial u}(x,\bar{u}_h)v_h \mathrm{d}x \quad \forall v_h \in \mathbb{V}_h,
\end{equation}
where $\bar{u}_h = \mathcal{S}_h \bar{q}_h$.
\end{theorem}
\begin{proof}
The proof follows the same arguments as in Theorem \ref{thm:first_opt_cond}.
\end{proof}

Analogous to the continuous case, we obtain projection formulas, but now on a discrete level.

\begin{theorem}[discrete projection formulas]
\label{thm:discrete_projection_formulas}
If $\bar{q}_h \in \mathbb{Q}_{ad,h}$ and $\bar{\eta}_h \in \partial j_h(\bar{q}_h)$ are as in the statement of Theorem \ref{thm:first_opt_cond_discrete}, then for every $T \in \T_h$, we have the formulas
\begin{align}
\label{eq:charac_qh}
&
\bar{q}_{h}|_{T} = \Pi_{[\alpha ,\beta]}
\left[
-\dfrac{1}{\lambda}\left(\dfrac{1}{|T|}\int_{T}\bar{p}_{h}\mathrm{d}x + \mu \bar{\eta}_{h}|_{T}\right)
\right],
\\
\label{eq:charac_qh_sparse}
&
\bar{q}_{h}|_{T} = 0
\Leftrightarrow
\dfrac{1}{|T|}\left|\int_{T}\bar{p}_{h} \mathrm{d}x \right| \leq \mu,
\\
\label{eq:charac_etah}
&
\bar{\eta}_{h}|_{T} = \Pi_{[-1,1]}\left(-\dfrac{1}{\mu|T|}\int_{T}\bar{p}_{h} \mathrm{d}x \right).
\end{align}
In particular, the subgradient $\bar{\eta}_h \in \partial j_h(\bar{q}_h)$ is uniquely determined.
\end{theorem}
\begin{proof}
From the variational inequality \eqref{eq:var_ineq_discrete} we have that, for every $q_T \in [\alpha,\beta]$,
\[
\left(
\int_{T}\bar{p}_h \mathrm{d}x
+
\left[
\lambda \bar{q}_h|_{T}
+
\mu \bar{\eta}_h|_{T}
\right] |T|
\right)
(q_T - \bar{q}_h|_{T})
\geq 0.
\]
The projection formula \eqref{eq:charac_qh} can be derived from this inequality. The sparse property \eqref{eq:charac_qh_sparse} and the projection formula \eqref{eq:charac_etah} are derived in
\cite[equivalence (4.4b)]{MR3023751} and \cite[formula (4.4c)]{MR3023751}, respectively.
\end{proof}

We conclude this section with the following error bounds for the discretization of the adjoint equation and the corresponding subdifferential variables by finite elements. These error bounds hold under the assumption that discrete solutions $u_{h}$ of the problem \eqref{eq:discrete_st_eq_FD} are uniformly bounded in $L^{\infty}(\Omega)$, i.e.,
\begin{equation}\label{eq:uniform_bound_Linf_uh_FD}
\exists C > 0: \quad \| u_{h} \|_{L^{\infty}(\Omega)} \leq C \quad \forall h > 0.
\end{equation}

To present a proof of some of the aforementioned error bounds, we introduce $\mathcal{P}_{h}: L^2(\Omega) \rightarrow \mathbb{Q}_{h}$, the orthogonal projection operator onto piecewise constant functions.

\begin{theorem}[error bounds]\label{thm:error_bounds}
Let $\bar{q}$ be a local solution to \eqref{eq:weak_min_problem}--\eqref{eq:weak_st_eq}
and let $\bar{\eta} \in \partial j(\bar{q})$ be as in Theorems \ref{thm:first_opt_cond} and \ref{thm:projection_formulas}. Let $\bar{p}$ be the solution to \eqref{eq:adjoint_eq} with $u$ replaced by $\bar u = \mathcal{S} \bar{q}$. Let $\bar{q}_h$ be a local solution to the fully discrete scheme and let $\bar{\eta}_h \in \partial j_h(\bar{q}_h)$ be as in Theorems \ref{thm:first_opt_cond_discrete} and \ref{thm:discrete_projection_formulas}. Let $\bar{p}_h$ be the solution to \eqref{eq:adjoint_eq_discrete}. Let us assume that \ref{A1}--\ref{A3}, \ref{B1}--\ref{B2}, \ref{C2}, and \eqref{eq:uniform_bound_Linf_uh_FD} hold. Then, we have
\begin{align}
\label{eq:error_est_seminorm_st_adj} 
% \| \bar{u} - \bar{u}_{h} \|_{s}
% +
\| \bar{p} - \bar{p}_{h} \|_{s} & \lesssim h^{\gamma}|\log h|^{\varphi} + \|\bar{q} - \bar{q}_{h}\|_{L^{2}(\Omega)},
\quad
\gamma = \min \{s,\tfrac{1}{2} \}
\\
\label{eq:error_est_L2norm_st_adj_subdiff}
\| \bar{p} - \bar{p}_{h} \|_{L^{2}(\Omega)}
+
\|\bar{\eta} - \bar{\eta}_{h}\|_{L^{2}(\Omega)} & \lesssim h^{2\gamma}|\log h|^{2\varphi} + \|\bar{q} - \bar{q}_{h}\|_{L^{2}(\Omega)},
\end{align}
where $\varphi = \nu$ if $s\neq \frac{1}{2}$, $\varphi = 1 +\nu$ if $s=\frac{1}{2}$, and $\nu \geq \frac{1}{2}$ is the constant in Theorem \ref{thm:sobolev_reg}.
\end{theorem}
\begin{proof}
The error estimates for $\bar{p} - \bar{p}_{h}$ in $\tilde{H}^s(\Omega)$ and $L^2(\Omega)$ can be found in \cite[Theorem 6.2, estimate (6.17)]{MR4599045} and \cite[Theorem 6.2, estimate (6.18)]{MR4599045}, respectively. It is worth noting that in the proof of \cite[Theorem 6.2]{MR4599045} it is assumed that $\partial L / \partial u$ is locally Lipschitz with respect to the second variable. This assumption can be removed at the expense of having the assumption on the second derivative in \ref{C2}. It remains to prove the estimate for $\bar{\eta} - \bar{\eta}_{h}$ in $L^2(\Omega)$. To do this, we use the projection formulas for $\bar{\eta}$ and $\bar{\eta}_h$, which are given in  \eqref{eq:charac_eta} and \eqref{eq:charac_etah}, respectively, and arrive at
\begin{equation}
\label{eq:eta-eta_h}
 \| \bar{\eta} - \bar{\eta}_h\|_{L^2(\Omega)} \leq \mu^{-1} \|\bar{p} - \mathcal{P}_h \bar{p}_h \|_{L^2(\Omega)}.
\end{equation}
A simple application of a triangle inequality thus results in $\| \bar{\eta} - \bar{\eta}_h\|_{L^2(\Omega)} \leq \mu^{-1} \|\bar{p} - \mathcal{P}_h \bar{p} \|_{L^2(\Omega)} + C \mu^{-1} \| \bar{p} - \bar{p}_h \|_{L^2(\Omega)}$ using the fact that $\mathcal{P}_h$ is stable in $L^2(\Omega)$. The regularity property for $\bar{p}$ derived in Theorem \ref{thm:reg_prop_u_p}, namely, $\bar{p}\in H^{s + \kappa - \varepsilon}(\Omega)$, where $\kappa$ and $\varepsilon$ are the same as in the stament of Theorem \ref{thm:reg_prop_u_p}, and a basic error estimate for the orthogonal projection $\mathcal{P}_h$ yield
\begin{equation}
\label{eq:p-P_hp}
 \|\bar{p} - \mathcal{P}_h \bar{p} \|_{L^2(\Omega)} \lesssim h^{2\gamma} | \log h|^{\nu} \Lambda(L,\bar{u}),
 \qquad
 \gamma = \min \{s,\tfrac{1}{2} \}.
\end{equation}
This and the derived error estimate for $\bar{p} - \bar{p}_h $ in $L^2(\Omega)$ allow us to complete the error estimate \eqref{eq:error_est_L2norm_st_adj_subdiff}. This completes the proof.
\end{proof}

\subsection{Convergence of discretizations}

We begin this section with a convergence result that guarantees that a sequence of discrete global solutions $\{ \bar{q}_h\}_{h>0}$ contains subsequences that converge to global solutions of problem \eqref{eq:weak_min_problem}--\eqref{eq:weak_st_eq} as $h\rightarrow 0$.

\begin{theorem}[convergence of discrete solutions]
 Let us assume that \ref{A1}--\ref{A3}, \ref{B1}--\ref{B2}, and \ref{C2} hold. Let $h > 0$ and let  $ \bar{q}_h \in \mathbb{Q}_{ad,h}$ be a global solution to the fully discrete control problem. Then, there exist nonrelabeled subsequences of $\{ \bar{q}_h\}_{h>0}$ such that $\bar{q}_{h} \mathrel{\ensurestackMath{\stackon[1pt]{\rightharpoonup}{\scriptstyle\ast}}} \bar{q}$ in the weak$^{*}$ topology of $L^{\infty}(\Omega)$ as $h \rightarrow 0$ and $\bar q$ corresponds to a global solution of the continuous control problem \eqref{eq:weak_min_problem}--\eqref{eq:weak_st_eq}. Furthermore,
\begin{equation}
\lim_{h \rightarrow 0}
\|\bar{q} - \bar{q}_h\|_{L^2(\Omega)} = 0,
\quad
% \lim_{h \rightarrow 0}
% \|\bar{\eta} - \bar{\eta}_h\|_{L^2(\Omega)} = 0,
% \quad
\lim_{h \rightarrow 0} \mathfrak{j}_h( \bar{q}_h) = \mathfrak{j}( \bar{q}).
\label{eq:convergence_first_result}
\end{equation}
If, in addition, \eqref{eq:uniform_bound_Linf_uh_FD} holds, then
\begin{equation}
\|\bar{\eta} - \bar{\eta}_h\|_{L^2(\Omega)} \rightarrow 0
\label{eq:convergence_first_result_eta}
\end{equation}
as $h \rightarrow 0$.
\label{thm:convergence_first_theorem}
\end{theorem}
\begin{proof}
Since $\{\bar q_h\}_{h>0} \subset \mathbb{Q}_{ad,h}$ is uniformly bounded in $L^{\infty}(\Omega)$, we can extract a nonrelabeled subsequence such that $\bar{q}_{h} \mathrel{\ensurestackMath{\stackon[1pt]{\rightharpoonup}{\scriptstyle\ast}}} \bar{q}$ in the weak$^{*}$ topology of $L^{\infty}(\Omega)$ as $h \rightarrow 0$. In the following, we prove that the limit $\bar{q}$ is a global solution to \eqref{eq:weak_min_problem}--\eqref{eq:weak_st_eq}.

Let $\bar{\mathsf{q}} \in \mathbb{Q}_{ad}$ be a global solution to \eqref{eq:weak_min_problem}--\eqref{eq:weak_st_eq} and define $\mathsf{q}_h = \mathcal{P}_h \bar{\mathsf{q}} \in \mathbb{Q}_h$. We note that according to the definition of $\mathcal{P}_h$, $\mathsf{q}_h$ is such that $\mathsf{q}_h|_{T}= \int_{T} \bar{\mathsf{q}}(x) \mathrm{d}x /|T|$ for each $T \in \T_h$. This immediately guarantees that $\mathsf{q}_h \in \mathbb{Q}_{ad,h}$, because $\bar{\mathsf{q}} \in \mathbb{Q}_{ad}$. We now take advantage of the fact that $\bar{\mathsf{q}} \in H^{s+\kappa - \varepsilon}(\Omega)$, where $\kappa$ and $\varepsilon$ are as in the statement of Theorem \ref{thm:reg_prop_u_p}, and a basic error estimate for $\mathcal{P}_h$ to deduce that $\| \bar{\mathsf{q}} - \mathsf{q}_h\|_{L^2(\Omega)} \rightarrow 0$ as $h \rightarrow 0$. We now use the global optimality of $\bar{\mathsf{q}} \in \mathbb{Q}_{ad}$ in conjunction with the fact that $\bar{q} \in \mathbb{Q}_{ad}$, a convergence result based on Lemma \ref{lemma:conv_disc_st_eq} and the fact that $\bar{q}_{h} \mathrel{\ensurestackMath{\stackon[1pt]{\rightharpoonup}{\scriptstyle\ast}}} \bar{q}$ in the weak$^{*}$ topology of $L^{\infty}(\Omega)$ as $h \rightarrow 0$, the global discrete optimality of $\bar{q}_h \in \mathbb{Q}_{ad,h}$ combined with the fact that $\mathsf{q}_h \in \mathbb{Q}_{ad,h}$, and another convergence result based on Lemma \ref{lemma:conv_disc_st_eq} and the fact that $\| \bar{\mathsf{q}} - \mathsf{q}_h\|_{L^2(\Omega)} \rightarrow 0$ as $h \rightarrow 0$ to obtain
\[
\mathfrak{j}(\bar{\mathsf{q}})
\leq
\mathfrak{j}(\bar q)
\leq
\liminf_{h \rightarrow 0} \mathfrak{j}_h(\bar q_h)
\leq
\limsup_{h \rightarrow 0} \mathfrak{j}_h(\bar q_h)
\leq
\limsup_{h \rightarrow 0} \mathfrak{j}_h(\mathsf{q}_h)
=
\mathfrak{j}(\bar{\mathsf{q}}).
\]
As a result, we have obtained that $\mathfrak{j}(\bar{\mathsf{q}})
\leq
\mathfrak{j}(\bar q) \leq \mathfrak{j}(\bar{\mathsf{q}})$, which guarantees the global optimality of $\bar{q}$. We also proved that $\mathfrak{j}_h(\bar q_h) \rightarrow \mathfrak{j}(\bar q)$ as $h \rightarrow 0$.

We now prove that $\|\bar{q} - \bar{q}_h\|_{L^2(\Omega)} \rightarrow 0$ as $h \rightarrow 0$. Due to Lemma \ref{lemma:conv_disc_st_eq}, we have that $\bar{u}_h = \mathcal{S}_h \bar{q}_h \rightarrow \bar{u} = \mathcal{S} \bar{q}$ in $L^{\mathfrak{r}}(\Omega)$ for every $\mathfrak{r} \leq 2d/(d-2s)$. From this follows
$
|
\int_{\Omega}
(L(x,\bar{u}_h) - L(x,\bar{u})) \mathrm{d}x
|
\rightarrow 0
$
as $h \rightarrow 0$. Since $\mathfrak{j}_h(\bar q_h) \rightarrow \mathfrak{j}(\bar q)$ as $h \rightarrow 0$, we obtain
\[
\frac{\lambda}{2} \| \bar{q}_h \|^2_{L^2(\Omega)} + \mu \| \bar{q}_h \|_{L^1(\Omega)}
\rightarrow
\frac{\lambda}{2} \| \bar{q} \|^2_{L^2(\Omega)} + \mu \| \bar{q} \|_{L^1(\Omega)},
\quad
h \rightarrow 0.
\]
This and the weak convergence $\bar{q}_{h} \rightharpoonup \bar{q}$ in $L^{2}(\Omega)$ as $h \rightarrow 0$ guarantee the strong one $\|\bar{q} - \bar{q}_h\|_{L^2(\Omega)} \rightarrow 0$ as $h \rightarrow 0$. The convergence of the term $\|\bar{\eta} - \bar{\eta}_h\|_{L^2(\Omega)}$ follows immediately from Theorem \ref{thm:error_bounds}. This concludes the proof.
\end{proof}

To present the next result, we introduce $B_{\epsilon}(\bar{q}):= \{ q \in L^2(\Omega): \| \bar{q} - q \|_{L^2(\Omega)} \leq \epsilon \}$.

\begin{theorem}
Let the assumptions of Theorem \ref{thm:convergence_first_theorem} hold. If $\bar{q}$ is a strict local minimum of the control problem \eqref{eq:weak_min_problem}--\eqref{eq:weak_st_eq}, then there is a sequence $\{ \bar{q}_h\}_{0<h<h_{\dagger}}$ of local minima of the fully discrete control problems such that \eqref{eq:convergence_first_result} and \eqref{eq:convergence_first_result_eta} hold.
\label{thm:convergence_second_theorem}
\end{theorem}
\begin{proof}
Since $\bar{q}$ is a strict local minimum of  \eqref{eq:weak_min_problem}--\eqref{eq:weak_st_eq}, there exists $\epsilon >0$ such that $\bar{q}$ is the unique solution to the following problem: Find $\min\{ \mathfrak{j}(q): q \in \mathbb{Q}_{ad} \cap B_{\epsilon}(\bar{q}) \}$.

We now intoduce the following discrete problem for each $h>0$: Find $\min\{ \mathfrak{j}_h(q_h): q_h \in \mathbb{Q}_{ad,h} \cap B_{\epsilon}(\bar{q}) \}$. We note that there exists $h_{\star}>0$ so that for each $h \in (0,h_{\star})$ the discrete function $ \mathcal{P}_h \bar{q}$ belongs to $\mathbb{Q}_{ad,h} \cap B_{\epsilon}(\bar{q})$. Consequently, $\mathbb{Q}_{ad,h} \cap B_{\epsilon}(\bar{q})$ is not empty and the previously introduced discrete problem admits a solution.

Let $h \in (0,h_{\star})$ and let $\bar{q}_h \in \mathbb{Q}_{ad,h} \cap B_{\epsilon}(\bar{q})$ be a global solution of the aforementioned discrete problem. The arguments elaborated in the proof of Theorem \ref{thm:convergence_first_theorem} show the existence of a nonrelabeled subsequence of $\{ \bar{q}_h \}_{0<h<h_{\star}}$ such that it converges strongly in $L^{2}(\Omega)$ as $h \rightarrow 0$ to a global solution of the following problem: Find $\min\{ \mathfrak{j}(q): q \in \mathbb{Q}_{ad} \cap B_{\epsilon}(\bar{q}) \}$. As mentioned at the beginning of the proof, this problem admits a unique solution $\bar{q}$. Consequently, the whole sequence $\{ \bar{q}_h \}_{0<h<h_{\star}}$ must converge to $\bar{q}$ in $L^2(\Omega)$ as $h \rightarrow 0$. As a result, the constraint $\bar{q}_h \in \mathbb{Q}_{ad,h} \cap B_{\epsilon}(\bar{q})$ is not active for $h$ sufficiently small so that $\bar{q}_h$ solves the fully discrete scheme.
\end{proof}

\subsection{Error estimates} \label{sec:FD_error_estimates}

This section is dedicated to the derivation of error estimates. For this purpose, we let $\{\bar{q}_{h}\}_{h>0} \subset \mathbb{Q}_{ad,h}$ be a sequence of local  minima of the fully discrete optimal control problems such that $\|\bar{q} - \bar{q}_{h}\|_{L^{2}(\Omega)} \rightarrow 0$ as $h \rightarrow 0$, where $\bar{q}$ corresponds to a local solution of \eqref{eq:weak_min_problem}--\eqref{eq:weak_st_eq}; see Theorems \ref{thm:convergence_first_theorem} and \ref{thm:convergence_second_theorem}. The main goal of this section is to obtain an error estimate for $\bar{q} - \bar{q}_{h}$ in $L^{2}(\Omega)$, namely
\begin{equation}
\label{eq:error_est_ct_FD}
\| \bar{q} - \bar{q}_{h} \|_{L^{2}(\Omega)} \lesssim
h^{2\gamma}|\log h|^{2\varphi}
\quad
\forall h \leq h_{\bullet},
\quad
\gamma = \min\left\{ s,\tfrac{1}{2}\right\},
\end{equation}
where $\varphi = \nu$ if $s\neq \frac{1}{2}$, $\varphi = 1 +\nu$ if $s=\frac{1}{2}$, and $\nu \geq \frac{1}{2}$ is the constant in Theorem \ref{thm:sobolev_reg}.

The following result is helpful to obtain \eqref{eq:error_est_ct_FD}.

\begin{theorem}[instrumental error estimate]\label{thm:instrumental_result}
Let us assume that \ref{A1}--\ref{A3}, \ref{B1}--\ref{B2}, \ref{C2}, and \eqref{eq:uniform_bound_Linf_uh_FD} hold. Let $\bar{q} \in \mathbb{Q}_{ad}$ satisfy the second order optimality condition (i), or  equivalently (ii) in Theorem \ref{thm:equivalence}. If \eqref{eq:error_est_ct_FD} is false, then there exists $h_{\Box} > 0$ such that
\begin{equation}
\label{eq:instrumental_result}
\mathfrak{C}\|\bar{q} - \bar{q}_{h}\|^{2}_{L^{2}(\Omega)} \leq [F'(\bar{q}_{h}) - F'(\bar{q})](\bar{q}_{h} - \bar{q})
\quad
\forall h \in (0,h_{\Box}],
\end{equation}
where $\mathfrak{C} := 2^{-1}\delta$ and $\delta$ is the constant that appears in the item (ii) of Theorem \ref{thm:equivalence}.
\end{theorem}
\begin{proof}
In a first step, we invoke the $C^{2}$ regularity of $F$ in $L^{2}(\Omega)\cap L^{r}(\Omega)$, with $r > d/2s$ (cf. Proposition \ref{pro:diff_F}), the $L^{2}(\Omega)$-convergence of $\bar{q}_{h}$ to $\bar{q}$ as $h \rightarrow 0$, and the mean value theorem to conclude the existence of $\epsilon > 0$ and $h_{\epsilon} > 0$ such that
\begin{equation}\label{eq:inst_result_mean_value_thm}
[F'(\bar{q}_{h}) - F'(\bar{q})](\bar{q}_{h} - \bar{q}) \geq F''(\bar{q})(\bar{q}_{h} - \bar{q})^{2} - \dfrac{\delta}{2}\|\bar{q}_{h} - \bar{q}\|_{L^{2}(\Omega)}^{2} \quad \forall h \leq h_{\epsilon}.
\end{equation}
It therefore suffices to prove that $\bar{q}_{h} - \bar{q} \in C_{\bar{q}}^{\tau}$ for some $\tau > 0$ and for every $h \leq h_{\tau}$, where $h_{\tau} >0$, to apply item (ii) of Theorem \ref{thm:equivalence} and deduce \eqref{eq:instrumental_result} with $h_{\Box}:= \min \{ h_{\tau}, h_{\epsilon} \}$. Therefore, the rest of the proof is devoted to proving that $\bar{q}_{h} - \bar{q} \in C_{\bar{q}}^{\tau}$.

For each $h>0$, we define the function
$
v_{h} := (\bar{q}_{h} - \bar{q})/\|\bar{q}_{h} - \bar{q}\|_{L^{2}(\Omega)}.
$
Since $\{ v_{h} \}_{h>0}$ is uniformly bounded in $L^2(\Omega)$, there exists a nonrelabeled subsequence such that $v_{h} \rightharpoonup v$ in $L^{2}(\Omega)$ as $h \rightarrow 0$. We now prove the existence of $\tau > 0$ and $h_{\tau} > 0$ such that $v_h \in C_{\bar{q}}^{\tau}$ for every $h \leq h_{\tau}$. The fact that every $v_h$ satisfies the sign conditions \eqref{eq:sign_cond} is trivial. It is therefore sufficient to prove that
\begin{equation}\label{eq:ineq_cone_ctau_v}
F'(\bar{q})v_{h} + \mu j'(\bar{q};v_{h}) \leq \tau \qquad \forall h \leq h_{\tau}.
\end{equation}
Since \cite[Lemma 3.5]{MR3023751} guarantees that $F'(\bar{q})v_{h} + \mu j'(\bar{q};v_{h}) \geq 0$, which holds for every $h>0$ because $v_h \in L^2(\Omega)$ satisfies \eqref{eq:sign_cond}, we will obtain 
%the inequality 
\eqref{eq:ineq_cone_ctau_v} with the help of the limit
$
F'(\bar{q})v_{h} + \mu j'(\bar{q};v_{h}) \rightarrow 0
$
as $h \rightarrow 0$. This will be now the focus of the proof.

Let us first note that the arguments developed in the proof of \cite[Theorem 7.4]{MR4358465} based on the weak convergence $v_{h} \rightharpoonup v$ in $L^{2}(\Omega)$ as $h \rightarrow 0$ show that
\begin{equation}\label{eq:integral_equals_0}
\left\{ 
F'(\bar{q})v + \mu \int_{\Omega}\bar{\eta}v \mathrm{d}x 
\right \}
= 
\lim_{h \rightarrow 0} 
\left\{ 
F'(\bar{q}_h)v_{h} + \mu \int_{\Omega}\bar{\eta}_hv_h \mathrm{d}x 
\right\}
\leq 0.
\end{equation}
This is the place where we use that \eqref{eq:error_est_ct_FD} is false. On the other hand, $F'(\bar{q})v + \mu \int_{\Omega}\bar{\eta}v \mathrm{d}x = \int_{\Omega} [\bar{p} + \lambda \bar{q} + \mu \bar{\eta}] v \mathrm{d}x \geq 0$ as a consequence of the variational inequality \eqref{eq:var_ineq} and the fact that $v$ satisfies \eqref{eq:sign_cond}. The relation $F'(\bar{q})v + \mu \int_{\Omega}\bar{\eta}v \mathrm{d}x = 0$ can be derived from this.

We now prove that $j'(\bar q; v_h) \rightarrow \int_{\Omega} \bar{\eta} v \mathrm{d}x$ as $h \rightarrow 0$; recall that $j'(\bar{q}; \cdot)$ is defined in \eqref{eq:dir_der_g}. To do this, we first note that $v_{h} \rightharpoonup v$ in $L^{2}(\Omega)$ as $h \rightarrow 0$ yields
\begin{equation}\label{eq:limit_direc_deriv_j}
\lim_{h \rightarrow 0}\left\{\int_{\Omega^{+}_{\bar{q}}}v_{h} \mathrm{d}x - \int_{\Omega^{-}_{\bar{q}}}v_{h} \mathrm{d}x \right\} = \int_{\Omega^{+}_{\bar{q}}}v \mathrm{d}x - \int_{\Omega^{-}_{\bar{q}}}v \mathrm{d}x.
\end{equation}
We now study $\int_{\Omega^{0}_{\bar{q}}}|v_{h}| \mathrm{d}x$, where $\Omega^{0}_{\bar{q}} = \{ x \in \Omega: \bar{q}(x) = 0 \}$. In view of \eqref{eq:charac_q}, the set $\Omega^{0}_{\bar{q}}$ can be rewriten as $\Omega^{0}_{\bar{q}} = \{x \in \Omega: |\bar{p}(x)| \leq \mu \}$. We decompose this set as follows:
\begin{align*}
& \Omega^{0}_{\bar{q}} = \Omega^{+}_{\mu}  \cup \Omega^{-}_{\mu} \cup \Omega^{\mathrm{less}}_{\mu},
\qquad
\Omega^{+}_{\mu} := \{ x \in \Omega: \bar{p}(x) = \mu \},
\\
& \Omega^{-}_{\mu} := \{ x \in \Omega: \bar{p}(x) = -\mu \},
\qquad
\Omega^{\mathrm{less}}_{\mu} := \{ x \in \Omega: |\bar{p}(x)| < \mu \}.
\end{align*}
With these sets at hand, we can write the integral $\int_{\Omega^{0}_{\bar{q}}}|v_{h}| \mathrm{d}x$ as follows:
\begin{equation}
\label{eq:noname}
\int_{\Omega_{\bar{q}}^{0}}|v_{h}| \mathrm{d}x = \int_{\Omega_{\mu}^{\mathrm{less}}}|v_{h}| \mathrm{d}x + \int_{\Omega_{\mu}^{+}}|v_{h}| \mathrm{d}x + \int_{\Omega_{\mu}^{-}}|v_{h}| \mathrm{d}x := \mathfrak{J}_{h} + \mathfrak{K}_{h} + \mathfrak{L}_{h}.
\end{equation}
In the following, we proceed in three steps to examine the three preceding terms.

\emph{Step 1.} We study the limit value of $\mathfrak{J}_{h}$ as $h \rightarrow 0$. For this purpose, for each $h > 0$ and $T \in \mathscr{T}_{h}$, we introduce the average $\bar{p}_{h,T} := \int_{T}\bar{p}_{h} \mathrm{d}x/|T|$. We also define
\begin{align}
\label{eq:T1h_T2h}
& \mathscr{T}_{1,h} := \{ T \in \mathscr{T}_{h}: |\bar{p}_{h,T}| \leq \mu \}, 
\qquad 
\Omega_{1,h} := \Omega_{\mu}^{\text{less}} \cap \mathscr{T}_{1,h},
\\
& \mathscr{T}_{2,h} := \{ T \in \mathscr{T}_{h}: |\bar{p}_{h,T}| > \mu \}, 
\qquad
\Omega_{2,h} := \Omega_{\mu}^{\text{less}} \cap \mathscr{T}_{2,h}.
\label{eq:O1h_O2h}
\end{align}
In an abuse of notation, in what follows, by $\mathscr{T}_{1,h}$ and $\mathscr{T}_{2,h}$, we will indistinctively denote either these sets or the union of the triangles that comprise them.

We now note that for each $T \in \mathscr{T}_{1,h}$, $\bar{q}_h|_T = 0$ as a consequence of \eqref{eq:charac_qh_sparse}. Thus,
\[
 v_h(x)|_T = \frac{\bar{q}_{h}(x) - \bar{q}(x)}{\|\bar{q}_{h} - \bar{q}\|_{L^{2}(\Omega)}} \bigg|_T =  \frac{-\bar{q}(x)}{\|\bar{q}_{h} - \bar{q}\|_{L^{2}(\Omega)}} \bigg|_T = 0
\]
for a.e.~$x \in \Omega_{\mu}^{\text{less}} \cap T$ and for each $T \in \mathscr{T}_{1,h}$ due to \eqref{eq:charac_q}. As a result, we obtain
\[
 \int_{\Omega_{1,h}}|v_{h}| \mathrm{d}x = \sum_{T \in \mathscr{T}_{1,h}} \int_{T \cap \Omega_{\mu}^{\text{less}}} |v_{h}| \mathrm{d}x = 0
 \implies
 \mathfrak{J}_{h} = \int_{\Omega_{\mu}^{\text{less}}}|v_{h}| \mathrm{d}x = \int_{\Omega_{2,h}}|v_{h}| \mathrm{d}x.
\]
We now bound $\int_{\Omega_{2,h}}|v_{h}| \mathrm{d}x$. For this purpose we let $T \in \mathscr{T}_{2,h}$ and note that
\[
 0 < |\mu - |\bar{p}(x)|| < | |\bar{p}_{h,T}| - |\bar{p}(x)||
\]
for a.e.~$x \in \Omega_{\mu}^{\text{less}} \cap T$. Integrating over $\Omega_{\mu}^{\text{less}} \cap T$ results in
\begin{align*}
 0 
 \leq  
 \int_{T \cap \Omega_{\mu}^{\text{less}}}|\mu - 
 %|\bar{p}(x)|
 |\bar{p}||^{2} \mathrm{d}x 
 < 
 \int_{T}||\mathcal{P}_{h}\bar{p}_{h}| - 
 %|\bar{p}(x)||^{2} 
 |\bar{p}||^{2}\mathrm{d}x 
 \leq
 \|\mathcal{P}_{h}\bar{p}_{h} - \bar{p}\|^{2}_{L^{2}(T)}.
\end{align*}
Summing over all elements $T$ in $\mathscr{T}_{2,h}$ and using the arguments from the proof of Theorem \ref{thm:error_bounds} which guarantee that $\|\mathcal{P}_{h}\bar{p}_{h} - \bar{p}\|_{L^{2}(\Omega)} \rightarrow 0$ as $h \rightarrow 0$, it follows that $|\Omega_{2,h}| \rightarrow 0$ as $h \rightarrow 0$. This implies that
\begin{equation}\label{eq:limit_Ih}
\lim_{h \rightarrow 0}\mathfrak{J}_{h} = 0.
\end{equation}

\emph{Step 2.} We now study the limit value of $\mathfrak{K}_{h}$ as $h \rightarrow 0$. To do this, we define
\begin{align}
\label{eq:T3h_T4h}
& \mathscr{T}_{3,h} := \{ T \in \mathscr{T}_{h}: \bar{p}_{h,T} < -\mu \}, 
\qquad 
\Omega_{3,h} := \Omega_{\mu}^{+} \cap \mathscr{T}_{3,h},
\\
& \mathscr{T}_{4,h} := \{ T \in \mathscr{T}_{h}: \bar{p}_{h,T} > + \mu \}, 
\qquad
\Omega_{4,h} := \Omega_{\mu}^{+} \cap \mathscr{T}_{4,h},
\label{eq:O3h_O4h}
\end{align}
and $\Omega_{5,h}:= \Omega_{\mu}^{+} \cap \mathscr{T}_{1,h}$, where $\mathscr{T}_{1,h}$ is defined in \eqref{eq:T1h_T2h}. If we proceed as at the beginning of  \emph{Step 1}, we can deduce that $v_{h}(x) = 0$ for a.e.~$x \in \Omega_{5,h}$. Consequently,
\[
\mathfrak{K}_{h} = \int_{\Omega_{\mu}^+}|v_{h}| \mathrm{d}x = \int_{\Omega_{3,h}}|v_{h}| \mathrm{d}x + \int_{\Omega_{4,h}}|v_{h}| \mathrm{d}x.
\]
Let $T \in \mathscr{T}_{3,h}$. By definition, $\bar{p}_{h,T} < -\mu$. Recall that $\bar{p}(x) = \mu$ for a.e.~$x \in  \Omega_{\mu}^{+}$. Thus,
\[
0 \leq 4 \mu^{2}|T \cap \Omega_{\mu}^+| = \int_{T \cap \Omega_{\mu}^+ }(\mu - (-\mu))^{2} \mathrm{d}x < \int_{T}(\bar{p} - \mathcal{P}_{h}\bar{p}_{h})^{2} \mathrm{d}x = \|\bar{p} - \mathcal{P}_{h}\bar{p}_{h}\|^{2}_{L^{2}(T)}.
\]
Summing over all elements $T$ in $\mathscr{T}_{3,h}$ we obtain $|\Omega_{3,h}| \rightarrow 0$ as $h \rightarrow 0$. To complete the proof of this step, we let $T \in \mathscr{T}_{4,h}$. Note that $\bar{p}_{h,T} > \mu$. Using the projection formulas \eqref{eq:charac_qh} and \eqref{eq:charac_etah}, we can deduce that $\bar{q}_h|_T <0$. This implies that $v_h(x) < 0$ for a.e. $x \in T \cap \Omega_{\mu}^{+}$. As a result, $|v_h(x)| = - v_h(x)$ for a.e. $x \in T \cap \Omega_{\mu}^{+}$. We now invoke the weak convergence $v_{h} \rightharpoonup v$ in $L^{2}(\Omega)$ as $h \rightarrow 0$ to obtain
\begin{multline}
\label{eq:limit_IIh}
 \lim_{h \rightarrow 0}
 \mathfrak{K}_{h} 
 = 
 - \lim_{h \rightarrow 0}\int_{\Omega_{4,h}}v_{h} \mathrm{d}x
 = - \lim_{h \rightarrow 0}\int_{\Omega_{4,h}}v_{h} \mathrm{d}x
 - \lim_{h \rightarrow 0}\int_{\Omega_{3,h}}v_{h} \mathrm{d}x
 \\
 - \lim_{h \rightarrow 0}\int_{\Omega_{5,h}}v_{h} \mathrm{d}x
 = - \lim_{h \rightarrow 0}\int_{\Omega_{\mu}^{+}}v_{h} \mathrm{d}x
 = -\int_{\Omega_{\mu}^{+}}v \mathrm{d}x,
\end{multline}
because $\int_{\Omega_{3,h}}v_{h} \mathrm{d}x \rightarrow 0$ as $h \rightarrow 0$ and $\int_{\Omega_{5,h}}v_{h} \mathrm{d}x = 0$ for each $h>0$.

\emph{Step 3.} Using arguments similar to those in \emph{Step 2} we arrive at 
\begin{equation}\label{eq:limit_IIIh_inst_result}
\lim_{h \rightarrow 0}\mathfrak{L}_{h} = \int_{\Omega_{\mu}^{-}}v \mathrm{d}x.
\end{equation}

Collecting \eqref{eq:noname}, \eqref{eq:limit_Ih}, \eqref{eq:limit_IIh}, and \eqref{eq:limit_IIIh_inst_result} we conclude that
\begin{equation}\label{eq:int_om_0_barq}
\lim_{h \rightarrow 0}\int_{\Omega_{\bar{q}}^{0}}|v_{h}| \mathrm{d}x = -\int_{\Omega_{\mu}^{+}}v \mathrm{d}x + \int_{\Omega_{\mu}^{-}}v \mathrm{d}x. 
\end{equation}

After we have proved \eqref{eq:int_om_0_barq} we use \eqref{eq:limit_direc_deriv_j} and the fact that $\bar{\eta}(x) = 1$ for a.e.~$x\in\Omega^{+}_{\bar{q}}$ and $x \in \Omega_{\mu}^{-}$ and that $\bar{\eta}(x) = -1$ for a.e.~$x\in\Omega^{-}_{\bar{q}}$ and $x\in\Omega_{\mu}^{+}$ to deduce that
\begin{multline}
\lim_{h \rightarrow 0} j'(\bar{q};v_h) 
= 
\int_{\Omega^{+}_{\bar{q}}}v \mathrm{d}x 
- 
\int_{\Omega^{-}_{\bar{q}}}v \mathrm{d}x 
-
\int_{\Omega_{\mu}^{+}}v \mathrm{d}x 
+ \int_{\Omega_{\mu}^{-}}v \mathrm{d}x 
\\
= \int_{\Omega^{+}_{\bar{q}}} \bar{\eta} v \mathrm{d}x + \int_{\Omega^{-}_{\bar{q}}}\bar{\eta}v \mathrm{d}x 
+ \int_{\Omega_{\mu}^{+}}\bar{\eta}v \mathrm{d}x 
+ \int_{\Omega_{\mu}^{-}}\bar{\eta}v \mathrm{d}x 
= \int_{\Omega} \bar{\eta} v \mathrm{d}x. 
\label{eq:noname2}
\end{multline}
To obtain the last equality we have used that $\| v_h \|_{L^1(\Omega_{\mu}^{\text{less}})} \rightarrow 0$, which follows from \eqref{eq:limit_Ih}, and that $v_{h} \rightharpoonup v$ in $L^{2}(\Omega)$ as $h \rightarrow 0$ to conclude that $v = 0$ a.e.~in $\Omega_{\mu}^{\text{less}}$.

The limit \eqref{eq:noname2} and the relation \eqref{eq:integral_equals_0} allows us to conclude that
\[
\lim_{h \rightarrow 0}\left( F'(\bar{q})v_{h} + \mu j'(\bar{q};v_{h}) \right) = 0.
\]
Since $F'(\bar{q})v_{h} + \mu j'(\bar{q};v_{h}) \geq 0$ for each $h > 0$, we deduce the existence of $\tau > 0$ and $h_{\tau} > 0$ such that
$F'(\bar{q})v_{h} + \mu j'(\bar{q};v_{h}) \leq \tau$ for all $h \leq h_{\tau}$.
This concludes the proof.
\end{proof}

We now provide a proof for the main result of this section.

\begin{theorem}[error bound for the approximation of an optimal control]\label{thm:error_bound_control}
Let us assume that \ref{A1}--\ref{A3}, \ref{B1}--\ref{B2}, \ref{C2}, and \eqref{eq:uniform_bound_Linf_uh_FD} hold. Let $\bar{q} \in \mathbb{Q}_{ad}$ satisfy the second order optimality condition (i), or equivalently (ii) in Theorem \ref{thm:equivalence}. Then, there is $h_{\bullet} > 0$ so that the estimate \eqref{eq:error_est_ct_FD} holds.
\end{theorem}
\begin{proof}
We proceed by contradiction: If we assume that \eqref{eq:error_est_ct_FD} is false, then there exists $h_{\Box} > 0$ such that the estimate \eqref{eq:instrumental_result} of Theorem \ref{thm:instrumental_result} holds for every $h \in (0,h_{\Box}]$. Based on the instrumental error estimate
\eqref{eq:instrumental_result}, we use the continuous and discrete optimality conditions, \eqref{eq:var_ineq} and \eqref{eq:var_ineq_discrete}, respectively, to obtain \cite[ineq. (4.14)]{MR3023751}
\begin{multline*}
\mathfrak{C}\|\bar{q} - \bar{q}_{h}\|^{2}_{L^{2}(\Omega)}
\leq
[F_{h}'(\bar{q}_{h}) - F'(\bar{q}_{h})](\bar{q} - \bar{q}_{h})
+
[F_{h}'(\bar{q}_{h}) - F'(\bar{q})](\mathcal{P}_h\bar{q} - \bar{q})
\\
+ \left[F'(\bar{q})(\mathcal{P}_h\bar{q} - \bar{q}) + \mu \int_{\Omega}\bar{\eta}(\mathcal{P}_h\bar{q} - \bar{q}) \mathrm{d}x\right]
+
\mu\int_{\Omega}(\bar{\eta} - \bar{\eta}_{h})(\bar{q}_{h} - \bar{q})\mathrm{d}x
\\
+ \mu\int_{\Omega}(\bar{\eta}_{h} - \bar{\eta})(\mathcal{P}_h\bar{q} - \bar{q}) \mathrm{d}x
:=
\mathbf{I}_h + \mathbf{II}_h  + \mathbf{III}_h + \mathbf{IV}_h + \mathbf{V}_h,
\quad
\forall h \leq h_{\Box}.
\end{multline*}

In the following, we proceed in several steps and estimate each of the terms $\mathbf{I}_h$, $\mathbf{II}_h$, $\mathbf{III}_h$, $\mathbf{IV}_h$, and $\mathbf{V}_h$ individually.

\emph{Step 1.} We first control the term $\mathbf{IV}_h$. To do this, we first use that $\bar{\eta} \in \partial j(\bar{q})$ and the definition of the subgradient given in \eqref{def:subgrad} to obtain
\[
\int_{\Omega}\bar{\eta}(\bar{q}_{h} - \bar{q})\mathrm{d}x \leq \|\bar{q}_{h}\|_{L^{1}(\Omega)} - \|\bar{q}\|_{L^{1}(\Omega)}.
\]
Since %$\bar{q}_h, \bar{\eta}_h \in \mathbb{Q}_h$, 
$\bar{\eta}_{h} \in \partial j_{h}(\bar{q}_{h})$, the characterization in \cite[eq. (4.3)]{MR3023751} leads to the conclusion that
\[
\int_{\Omega}\bar{\eta}_{h}(\bar{q} - \bar{q}_{h})\mathrm{d}x
=
\int_{\Omega}\bar{\eta}_{h}\bar{q} \mathrm{d}x - \|\bar{q}_{h}\|_{L^{1}(\Omega)}
\leq \|\bar{q}\|_{L^{1}(\Omega)} - \|\bar{q}_{h}\|_{L^{1}(\Omega)}.
\]
Consequently, we can control the term $\mathbf{IV}_h$ as follows:
\begin{equation}\label{eq:error_est_ct_FD_termIV}
\mathbf{IV}_h
\leq
\mu \left(
\|\bar{q}_{h}\|_{L^{1}(\Omega)} - \|\bar{q}\|_{L^{1}(\Omega)}
+
\|\bar{q}\|_{L^{1}(\Omega)}
-
\|\bar{q}_{h}\|_{L^{1}(\Omega)}
\right) = 0.
\end{equation}

\emph{Step 2}. We estimate $\mathbf{III}_h$. To do so, we use the characterization of $F'(\bar{q})$ described in \eqref{eq:charac_1_der} and standard properties for the orthogonal projection operator $\mathcal{P}_h$ to obtain
\begin{align}\label{eq:error_est_ct_FD_termIII_1}
\begin{split}
\mathbf{III}_h & = \int_{\Omega}(\bar{p} + \lambda \bar{q} + \mu \bar{\eta} - \mathcal{P}_h\left(\bar{p}+ \lambda \bar{q} + \mu \bar{\eta}\right))(\mathcal{P}_h\bar{q} - \bar{q}) \mathrm{d}x
\\
& \leq \left(\|\bar{p} - \mathcal{P}_h\bar{p}\|_{L^{2}(\Omega)} + \mu\|\bar{\eta} - \mathcal{P}_h\bar{\eta}\|_{L^{2}(\Omega)}\right)
\|\bar{q} - \mathcal{P}_h\bar{q}\|_{L^{2}(\Omega)}.
\end{split}
\end{align}
Here, $\bar{p} \in \tilde{H}^{s}(\Omega) \cap L^{\infty}(\Omega)$ denotes the solution to \eqref{eq:adjoint_eq} with $u$ replaced by $\bar{u} = \mathcal{S}\bar{q}$. The control of the error $\|\bar{p} - \mathcal{P}_h\bar{p}\|_{L^{2}(\Omega)}$ follows from a standard error estimate for $\mathcal{P}_h$ in conjunction with the regularity property $\bar{p} \in H^{s+\kappa-\varepsilon}(\Omega)$ for all $\varepsilon \in (0,s)$, where $\kappa = \tfrac{1}{2}$ for $ \tfrac{1}{2} < s < 1$ and $\kappa = s - \varepsilon$ for $0 < s \leq \tfrac{1}{2}$ (see Theorem \ref{thm:reg_prop_u_p}). In fact, we have
\begin{equation}\label{eq:error_est_ct_FD_termIII_2}
\|\bar{p} - \mathcal{P}_h\bar{p}\|_{L^{2}(\Omega)} \lesssim h^{2\gamma}|\log h|^{\nu}\Lambda(L,\bar{u}), \qquad \gamma = \min\{s,\tfrac{1}{2}\},
\end{equation}
where $\nu = \tfrac{1}{2}$ for $\tfrac{1}{2} < s < 1$ and $\nu = \tfrac{1}{2} + \nu_{0}$ for $0 < s \leq \tfrac{1}{2}$, with $\nu_{0} = \nu_{0}(\Omega,d) > 0$. The term $\Lambda(L,\bar{u})$ is defined in \eqref{eq:Lambda_L_u}.
Similarly, we have
\begin{equation}\label{eq:error_est_ct_FD_termIII_3}
\|\bar{q} - \mathcal{P}_h\bar{q}\|_{L^{2}(\Omega)}
+
\|\bar{\eta} - \mathcal{P}_h\bar{\eta}\|_{L^{2}(\Omega)}
\lesssim
h^{2\gamma}|\log h|^{\nu}
\left[1 + \Lambda(L,\bar{u})\right].
\end{equation}
If we replace the estimates obtained in \eqref{eq:error_est_ct_FD_termIII_2} and \eqref{eq:error_est_ct_FD_termIII_3} into \eqref{eq:error_est_ct_FD_termIII_1}, we obtain
\begin{equation}\label{eq:error_est_ct_FD_termIII}
\mathbf{III}_h \lesssim h^{4\gamma}|\log h|^{2\nu}\left[1 + \Lambda(L,\bar{u})\right]^{2}.
\end{equation}

\emph{Step 3.} An estimate for the term $\mathbf{V}_h$ follows easily from the definition of $\mathcal{P}_h$ and the estimate \eqref{eq:error_est_ct_FD_termIII_3} derived in the previous step:
\begin{equation}\label{eq:error_est_ct_FD_termV}
\mathbf{V}_h = \mu\int_{\Omega}(\mathcal{P}_h\bar{\eta} - \bar{\eta})(\mathcal{P}_h\bar{q} - \bar{q}) \mathrm{d}x
\lesssim
h^{4\gamma}|\log h|^{2\nu}
\left[1 + \Lambda(L,\bar{u})\right]^{2}.
\end{equation}

\emph{Step 4}. The aim of this step is to estimate the term $\mathbf{I}_h$.
To accomplish this task, we introduce the variables $\hat{u} \in \tilde{H}^{s}(\Omega)$ and $\hat{p} \in \tilde{H}^{s}(\Omega)$, which solve, respectively,
\begin{equation*}
\mathcal{A}(\hat{u},v) + \int_{\Omega}a(x,\hat{u})v \mathrm{d}x  = \int_{\Omega}\bar{q}_{h}v \mathrm{d}x \quad \forall v \in \tilde{H}^{s}(\Omega)
\end{equation*}
and
\begin{equation*}
\mathcal{A}(\hat{p},v) + \int_{\Omega} \dfrac{\partial a}{\partial u}(x,\hat{u})\hat{p}v \mathrm{d}x
= \int_{\Omega}\dfrac{\partial L}{\partial u}(x,\hat{u})v \mathrm{d}x \quad \forall v \in \tilde{H}^{s}(\Omega).
\end{equation*}
With these variables, we can rewrite and estimate the term $\mathbf{I}_h$ as follows:
\begin{equation}\label{eq:error_est_ct_FD_termI_1}
\mathbf{I}_h = \int_{\Omega}
\left[
(\bar{p}_{h} + \lambda\bar{q}_{h}) - (\hat{p} + \lambda\bar{q}_{h})
\right]
(\bar{q} - \bar{q}_{h}) \mathrm{d}x \leq \dfrac{1}{\mathfrak{C}}\|\bar{p}_{h} - \hat{p}\|_{L^{2}(\Omega)}^{2} + \dfrac{\mathfrak{C}}{4}\|\bar{q} - \bar{q}_{h}\|_{L^{2}(\Omega)}^{2}.
\end{equation}
To obtain the last inequality, we used Young's inequality. Here, $\mathfrak{C} = 2^{-1}\delta$ is the constant that appears in Theorem \ref{thm:instrumental_result}.

The rest of this step is dedicated to bound the term $\|\bar{p}_{h} - \hat{p}\|_{L^{2}(\Omega)}$. For this purpose, we introduce $\tilde{p}$ as the solution to: Find $\tilde{p} \in \tilde{H}^{s}(\Omega)$ such that
\begin{align*}
\mathcal{A}(\tilde{p},v) + \int_{\Omega} \dfrac{\partial a}{\partial u}(x,\bar{u}_{h})\tilde{p}v \mathrm{d}x & = \int_{\Omega}\dfrac{\partial L}{\partial u}(x,\bar{u}_{h})v \mathrm{d}x \quad \forall v \in \tilde{H}^{s}(\Omega).
\end{align*}
We note that, in view of  \eqref{eq:uniform_bound_Linf_uh_FD} and the assumptions on $a$ and $L$ all terms in this weak formulation are well-posed. With the variable $\tilde{p}$ at hand, the triangle inequality in $L^2(\Omega)$ yields $
\|\bar{p}_{h} - \hat{p}\|_{L^{2}(\Omega)} \leq \|\bar{p}_{h} - \tilde{p}\|_{L^{2}(\Omega)} + \|\tilde{p} - \hat{p}\|_{L^{2}(\Omega)}.
$
To bound $\|\bar{p}_{h} - \tilde{p}\|_{L^{2}(\Omega)}$, we use that $\bar{p}_{h}$ corresponds to the finite element approximation of $\tilde{p}$ within $\mathbb{V}_h$. Indeed, an application of a suitable modification of Theorem \cite[Theorem 6.1]{MR4599045} yields
\begin{equation}\label{eq:error_est_ct_FD_termI_2}
\|\bar{p}_{h} - \hat{p}\|_{L^{2}(\Omega)} \lesssim h^{2\gamma}|\log h|^{2\varphi}\Lambda(L,\bar{u}_{h}).
\end{equation}
We note that the same arguments we used to derive the regularity results for $\bar{p}$ in Theorem \ref{thm:reg_prop_u_p} also apply to $\tilde{p}$ with a similar estimate. It remains to bound the term $\|\tilde{p} - \hat{p}\|_{L^{2}(\Omega)}$. To do this, we first  note that $\tilde{p} - \hat{p}$ is such that
% solves the following problem: Find $\tilde{p} - \hat{p} \in \tilde{H}^{s}(\Omega)$ such that
\begin{multline}
\tilde{p} - \hat{p} \in \tilde{H}^{s}(\Omega):
\quad
\mathcal{A}(\tilde{p} - \hat{p},v) + \int_{\Omega}\dfrac{\partial a}{\partial u}(x,\hat{u})(\tilde{p} - \hat{p})v \mathrm{d}x
\\
=
\int_{\Omega}\dfrac{\partial^{2} L}{\partial u^{2}}(x,u_{\theta})(\bar{u}_{h} - \hat{u})v \mathrm{d}x
+
\int_{\Omega} \dfrac{\partial^{2} a}{\partial u^{2}}(x,u_{\vartheta})(\hat{u}-\bar{u}_{h})\tilde{p}v \mathrm{d}x
\label{eq:tilde_p_minus_hat_p}
\end{multline}
for all $v \in \tilde{H}^{s}(\Omega)$, where $u_{\theta} := \hat{u} + \theta_h (\bar{u}_h - \hat{u})$ and $u_{\vartheta}:= \bar{u}_h + \vartheta_h(\hat{u} - \bar{u}_h)$ with $\theta_h, \vartheta_h \in (0,1)$. Given the assumptions \ref{C2} and \eqref{eq:uniform_bound_Linf_uh_FD}, the $L^{\infty}(\Omega)$-regularity of $\tilde{p}$, and H\"older's inequality we have that all terms in \eqref{eq:tilde_p_minus_hat_p} are well-defined. We now invoke a stability estimate for problem \eqref{eq:tilde_p_minus_hat_p},
assumptions \ref{A3} and \ref{C2}, the $L^{\infty}(\Omega)$-regularity of $\tilde{p}$, and the embedding $H^{s}(\Omega) \hookrightarrow L^{\mathfrak{r}}(\Omega)$, which holds for every $\mathfrak{r} \leq 2d/(d-2s)$, to derive the estimates
\begin{multline}
 \label{eq:error_est_ct_FD_termI_3}
\|\tilde{p} - \hat{p}\|_{L^{2}(\Omega)}
\lesssim
\|\tilde{p} - \hat{p}\|_{s}
\leq
\left\| \left(\dfrac{\partial^{2} L}{\partial u^{2}}(\cdot,u_{\theta})
-
\dfrac{\partial^{2} a}{\partial u^{2}}(\cdot,u_{\vartheta}) \tilde{p} \right)(\bar{u}_{h} - \hat{u}) \right\|_{H^{-s}(\Omega)}
\\
\lesssim \left( \left\| \dfrac{\partial^{2} L}{\partial u^{2}}(\cdot,u_{\theta}) \right\|_{L^{\frac{d}{s}}(\Omega)} + \|\tilde{p}\|_{L^{\infty}(\Omega)}\right)\|\bar{u}_{h} - \hat{u}\|_{L^{2}(\Omega)}
\lesssim h^{2\gamma}|\log h|^{2\varphi}\|\bar{q}_{h} - a(\cdot,0)\|_{L^{2}(\Omega)}.
\end{multline}
To obtain the last estimate, we used the fact that $\bar{u}_{h}$ corresponds to the finite element approximation of $\hat{u}$ within $\mathbb{V}_h$ and Theorem \ref{thm:error_estimates_frac_Lap}. If we replace the estimates \eqref{eq:error_est_ct_FD_termI_2} and \eqref{eq:error_est_ct_FD_termI_3} in \eqref{eq:error_est_ct_FD_termI_1}, we obtain that
\begin{equation}\label{eq:error_est_ct_FD_termI}
\mathbf{I}_h \leq C h^{4\gamma}|\log h|^{4\varphi}\left( \|\bar{q}_{h} - a(\cdot,0)\|_{L^{2}(\Omega)} + \Lambda(L,\bar{u}_{h}) \right)^{2} + \dfrac{\mathfrak{C}}{4}\|\bar{q} - \bar{q}_{h}\|_{L^{2}(\Omega)}^{2}.
\end{equation}

\emph{Step 5.} In this step, we bound $\mathbf{II}_h$. To do this, we use the definition of the orthogonal projection operator $\mathcal{P}_h$ and the regularity properties of $\bar{p}$ and $\bar{q}$ to obtain
\begin{align*}
\begin{split}
\mathbf{II}_h = \int_{\Omega}((\bar{p}_{h} +  \lambda\bar{q}_{h} - (\bar{p} + \lambda\bar{q}))(\mathcal{P}_h\bar{q} - \bar{q})\mathrm{d}x
= \int_{\Omega}(\bar{p}_{h} - \bar{p})(\mathcal{P}_h\bar{q} - \bar{q})\mathrm{d}x
\\
+ \lambda\|\mathcal{P}_h\bar{q} - \bar{q}\|_{L^{2}(\Omega)}^{2} \leq \|\bar{p}_{h} - \bar{p}\|_{L^{2}(\Omega)}\|\mathcal{P}_h\bar{q} - \bar{q}\|_{L^{2}(\Omega)} + \lambda\|\mathcal{P}_h\bar{q} - \bar{q}\|_{L^{2}(\Omega)}^{2} .
\end{split}
\end{align*}
Use the estimates \eqref{eq:error_est_ct_FD_termIII_3} and \eqref{eq:error_est_L2norm_st_adj_subdiff}, as well as Young's inequality, to obtain
\begin{equation}\label{eq:error_est_ct_FD_termII}
\mathbf{II}_h \leq C h^{4\gamma}|\log h|^{4\varphi}\left[1 + \Lambda(L,\bar{u})\right]^2 + \dfrac{\mathfrak{C}}{4}\|\bar{q} - \bar{q}_{h}\|^{2}_{L^{2}(\Omega)}.
\end{equation}

\emph{Step 6.} Through the collection of \eqref{eq:error_est_ct_FD_termIV}, \eqref{eq:error_est_ct_FD_termIII}, \eqref{eq:error_est_ct_FD_termV}, \eqref{eq:error_est_ct_FD_termI}, and \eqref{eq:error_est_ct_FD_termII}, we conclude that \eqref{eq:error_est_ct_FD} holds, which is a contradiction. This concludes the proof.
\end{proof}

As a corollary, we present the following estimate for $\bar{\eta} - \bar{\eta}_h$.

\begin{corollary}[error bound for the approximation of an optimal subgradient]
In the framework of Theorem \ref{thm:error_bound_control}, we have the error bound
\begin{equation}
\label{eq:error_est_eta_FD}
\| \bar{\eta} - \bar{\eta}_{h} \|_{L^{2}(\Omega)} \lesssim
h^{2\gamma}|\log h|^{2\varphi}
\quad
\forall h \leq h_{\bullet},
\quad
\gamma = \min\left\{ s,\tfrac{1}{2}\right\},
\end{equation}
where $\varphi = \nu$ if $s\neq \frac{1}{2}$, $\varphi = 1 +\nu$ if $s=\frac{1}{2}$, and $\nu \geq \frac{1}{2}$ is the constant in Theorem \ref{thm:sobolev_reg}.
\end{corollary}
\begin{proof}
The error bound is an immediate consequence of Theorems \ref{thm:error_bounds} and \ref{thm:error_bound_control}.
\end{proof}

\section{A semidiscrete scheme for the optimal control problem}\label{sec:SD_scheme}

In the following, we propose a semidiscretization strategy based on the variational discretization approach \cite{MR2122182}. Here, only the state space is discretized (the control space is not discretized). The semidiscrete approach is as follows: Find $\min J(u_{h},\mathsf{q})$ subject to
\begin{equation}\label{eq:discrete_st_eq_SD}
\mathcal{A}(u_{h},v_{h}) + \int_{\Omega}a(x,u_{h})v_{h} \mathrm{d}x = \int_{\Omega}\mathsf{q}v_{h}\mathrm{d}x \quad \forall v_{h} \in \mathbb{V}_{h},
\end{equation}
and the control constraints $\mathsf{q} \in \mathbb{Q}_{ad}$. Standard arguments show that there is at least one optimal solution to this problem. Furthermore, as in Theorem \ref{thm:first_opt_cond}, it can be proved that if $\bar{\mathsf{q}} \in \mathbb{Q}_{ad}$ denotes a local solution, then there exists $\bar{\upeta} \in \partial j(\bar{\mathsf{q}})$ such that
% satisfies the following first-order optimality conditions:
\begin{equation}\label{eq:first_opt_cond_SD}
\int_{\Omega}(\bar{p}_{h} + \lambda \bar{\mathsf{q}} + \mu \bar{\upeta})(q - \bar{\mathsf{q}}) \mathrm{d}x \geq 0 \qquad \forall q \in \mathbb{Q}_{ad}.
\end{equation}
Here, $\bar{p}_{h}$ solves the problem \eqref{eq:adjoint_eq_discrete}, where $\bar{u}_{h}$ corresponds to the solution of \eqref{eq:discrete_st_eq_SD} with $\mathsf{q}$ replaced by $\bar{\mathsf{q}}$. As in Theorem \ref{thm:projection_formulas}, the following projection formulas can be derived for every $x \in \Omega$:
\begin{align}\label{eq:charact_var_control}
& \bar{\mathsf{q}}(x) = \Pi_{[\alpha , \beta]}\left(-\lambda^{-1}(\bar{p}_{h}(x) + \mu\bar{\upeta}(x)) \right), \qquad \bar{\mathsf{q}}(x) = 0 \Leftrightarrow |\bar{p}_{h}(x)| \leq \mu, \\
& \bar{\upeta}(x) = \Pi_{[-1,1]}\left(-\mu^{-1}\bar{p}_{h}(x)\right).
\end{align}
Since $\bar{\mathsf{q}}$ and $\bar{\upeta}$ implicitly depend on $h$, we will use the notation $\bar{\mathsf{q}}_{h}$ and $\bar{\upeta}_{h}$ in the following. Assuming that discrete solutions $\bar{u}_{h}$ of \eqref{eq:discrete_st_eq_SD} are uniformly bounded in $L^{\infty}(\Omega)$ and that \ref{A1}--\ref{A3}, \ref{B1}--\ref{B2}, and \ref{C2} hold, we can provide error bounds for the approximation error of the adjoint state and subdifferential variables. The error estimate for the latter is simpler than that in Theorem \ref{thm:error_bounds} because of the bound
$
 \| \bar{\eta} - \bar{\upeta}_h \|_{L^2(\Omega)} \leq \mu^{-1} \| \bar{p} - \bar{p}_h \|_{L^2(\Omega)}.
$
Furthermore, minor adjustments in the proofs of the Theorems \ref{thm:convergence_first_theorem} and \ref{thm:convergence_second_theorem} lead to the following convergence results.
\begin{itemize}
\item Let $h > 0$ and let $\bar{\mathsf{q}}_{h} \in \mathbb{Q}_{ad}$ be a global solution of the semidiscrete scheme. Then, there exist nonrelabeled subsequences of $\{\bar{\mathsf{q}}_{h}\}_{h > 0}$ such that $\bar{\mathsf{q}}_{h} \mathrel{\ensurestackMath{\stackon[1pt]{\rightharpoonup}{\scriptstyle\ast}}} \bar{q}$ in $L^{\infty}(\Omega)$ as $h \rightarrow 0$ and $\bar{q}$ corresponds to a global solution to \eqref{eq:weak_min_problem}--\eqref{eq:weak_st_eq}. In addition, the convergence results \eqref{eq:convergence_first_result} and \eqref{eq:convergence_first_result_eta} hold.
\item If $\bar{q} \in \mathbb{Q}_{ad}$ is a strict local minimum of the control problem \eqref{eq:weak_min_problem}--\eqref{eq:weak_st_eq}, then there exists a sequence of local minima $\{\bar{\mathsf{q}}_{h}\}_{0 < h < h_{\dagger}}$ of the semidiscrete scheme such that \eqref{eq:convergence_first_result} and \eqref{eq:convergence_first_result_eta} hold.
\end{itemize}

We now derive the error bound for the semidiscrete scheme given in \eqref{eq:error_control_var}. For this purpose, we let $\{\bar{\mathsf{q}}_{h}\} \subset \mathbb{Q}_{ad}$ be a sequence of local minima of such a scheme such that $\bar{\mathsf{q}}_{h} \rightarrow \bar{q}$ in $L^{2}(\Omega)$ as $h \rightarrow 0$, where $\bar{q} \in \mathbb{Q}_{ad}$ is a local solution to \eqref{eq:weak_min_problem}--\eqref{eq:weak_st_eq}. In a first step, we provide an instrumental result that is analogous to Theorem \ref{thm:instrumental_result}.

\begin{theorem}[instrumental error estimate]\label{thm:inst_result_SD}
Let us assume that \ref{A1}--\ref{A3}, \ref{B1}--\ref{B2}, \ref{C2} and \eqref{eq:uniform_bound_Linf_uh_FD} hold. Let $\bar{q} \in \mathbb{Q}_{ad}$ satisfy the second order optimality condition (i), or equivalently (ii) in Theorem \ref{thm:equivalence}. Then, there exists $h_{*} > 0$ such that
\begin{equation}
\label{eq:instrumental_result_SD}
\mathfrak{C}\|\bar{q} - \bar{\mathsf{q}}_{h}\|^{2}_{L^{2}(\Omega)} \leq [F'(\bar{\mathsf{q}}_{h}) - F'(\bar{q})](\bar{\mathsf{q}}_{h} - \bar{q})
\quad
\forall h \in (0,h_{*}],
\end{equation}
where $\mathfrak{C} = 2^{-1}\delta$ and $\delta$ is the constant that appears in the item (ii) of Theorem \ref{thm:equivalence}.
\end{theorem}
\begin{proof}
The proof follows analogous arguments as in the proof of Theorem \ref{thm:instrumental_result}. The main modifications are the redefinitions of the sets
\begin{equation*}
\begin{aligned}
& \Omega_{1,h} := \{ x \in \Omega_{\mu}^{\text{less}} :  |\bar{p}_{h}(x)| \leq \mu\},
\qquad
\Omega_{2,h}:= \{ x \in \Omega_{\mu}^{\text{less}}: |\bar{p}_{h}(x)| > +\mu \},
\\
& \Omega_{3,h} := \{ x \in \Omega_{\mu}^{+} :  \bar{p}_{h}(x) < -\mu\},
\qquad
\Omega_{4,h}:= \{ x \in \Omega_{\mu}^{+}: \bar{p}_{h}(x) > +\mu \},
\end{aligned}
\end{equation*}
and $\Omega_{5,h} := \{x \in \Omega_{\mu}^{+} :  |\bar{p}_{h}(x)| \leq \mu\}$. For the sake of simplicity, we skip the details.
\end{proof}

We now derive the main result of this section.

\begin{theorem}[error bound for the approximation of an optimal control]\label{thm:error_bound_control_SD}
Let us assume that \ref{A1}--\ref{A3}, \ref{B1}--\ref{B2}, \ref{C2}, and \eqref{eq:uniform_bound_Linf_uh_FD} hold. Let $\bar{q} \in \mathbb{Q}_{ad}$ satisfy the second order optimality condition (i), or equivalently (ii) in Theorem \ref{thm:equivalence}. Then, there exists $h_{\bullet} > 0$ such that
\begin{equation}\label{eq:error_control_var}
\| \bar{q} - \bar{\mathsf{q}}_{h} \|_{L^{2}(\Omega)} \lesssim
h^{2\gamma}|\log h|^{2\varphi}
\quad
\forall h \leq h_{\bullet}
\quad
\gamma = \min\left\{ s,\tfrac{1}{2}\right\},
\end{equation}
where $\varphi = \nu$ if $s\neq \frac{1}{2}$, $\varphi = 1 +\nu$ if $s=\frac{1}{2}$, and $\nu \geq \frac{1}{2}$ is the constant in Theorem \ref{thm:sobolev_reg}.
\end{theorem}
\begin{proof}
We proceed as in the proof of Theorem \ref{thm:error_bound_control} and use the instrumental error bound \eqref{eq:instrumental_result_SD} and the variational inequalities \eqref{eq:var_ineq} and \eqref{eq:first_opt_cond_SD} with $q = \bar{\mathsf{q}}_{h}$ and $q = \bar{q}$, respectively, to obtain (cf. \cite[Theorem 5.1]{MR3023751})
\begin{equation*}
\mathfrak{C}\|\bar{q} - \bar{\mathsf{q}}_{h}\|^{2}_{L^{2}(\Omega)}
\leq
[F_{h}'(\bar{\mathsf{q}}_{h}) - F'(\bar{\mathsf{q}}_{h})](\bar{q} - \bar{\mathsf{q}}_{h})
+
\mu\int_{\Omega}(\bar{\eta} - \bar{\upeta}_{h})(\bar{\mathsf{q}}_{h} - \bar{q})\mathrm{d}x \quad \forall h \leq h_{*}.
\end{equation*}
We immediately notice that the first and second terms on the right-hand side of the previous inequality correspond to $\mathbf{I}_{h}$ and $\mathbf{IV}_{h}$, respectively, from the proof of Theorem \ref{thm:error_bound_control}. These terms, $\mathbf{I}_{h}$ and $\mathbf{IV}_{h}$, are estimated in \eqref{eq:error_est_ct_FD_termI} and \eqref{eq:error_est_ct_FD_termIV}, respectively.
% To estimate $\mathbf{IV}_{h}$, we use the monotonicity property \eqref{eq:monot_subdiff}: $\mu\int_{\Omega}(\bar{\eta} - \bar{\upeta}_{h})(\bar{\mathsf{q}}_{h} - \bar{q})\mathrm{d}x \leq 0$. The estimation of $\mathbf{I}_{h}$ follows directly from  \eqref{eq:error_est_ct_FD_termI}.
% This concludes the proof.
\end{proof}

\section{Numerical examples}\label{sec:numerical_exp}

We present a numerical experiment that illustrates the performance of the fully and 
%semi discrete 
semidiscrete methods presented Sections \ref{sec:FD_scheme} and \ref{sec:SD_scheme}, respectively, when used to approximate a solution of the control problem \eqref{eq:weak_min_problem}--\eqref{eq:weak_st_eq}. A MATLAB implementation is used for the experiment, and the methods are solved using a semi-smooth Newton method.

The setting of the experiment is as follows: we set $d = 2$, $\Omega = B(0,1)$, and $\lambda = 1$, where $B(0,1)$ denotes the unit disc. We let $a(\cdot,u) = u^{3}$ and $L(\cdot,u) = (u - u_{\Omega})^{2}/2$, where $u_{\Omega}$ is such that the exact optimal state and the optimal adjoint state are
\begin{equation}
	\bar{u}(x) = \bar{p}(x) = ( 2^{2s}\Gamma^{2}\left(1 + s\right))^{-1}(1 - |x|^{2})^{s}_{+},
	\qquad t_{+} = \max\{0,t\}.
	\label{eq:u_and_p}
\end{equation}
We also consider $a = -1$, $b = 1$, and $s \in \{0.2 , 0.4 , 0.6 , 0.8 \}$. Additionally, for $s \leq 0.5$, we set $\mu = 0.6$, and for $s > 0.5$, we choose $\mu = 0.25$.

%%%%%%%%%%%%%%%%%%%%%%%%%%%%%%%%%%%%%%%%
%%%%%%%%%%%%%%%%%%%%%%%%%%%%%%%%%%%%%%%%
%%%%%%%%                   %%%%%%%%%%%%%
%%%%%%%% EXAMPLE 1 - FIG.1 %%%%%%%%%%%%%
%%%%%%%%                   %%%%%%%%%%%%%
%%%%%%%%%%%%%%%%%%%%%%%%%%%%%%%%%%%%%%%%
%%%%%%%%%%%%%%%%%%%%%%%%%%%%%%%%%%%%%%%%

\begin{figure}[!ht]
\centering
		\includegraphics[trim={0 0 0 0},clip,width=12.51cm,height=10.71cm,scale=0.4]{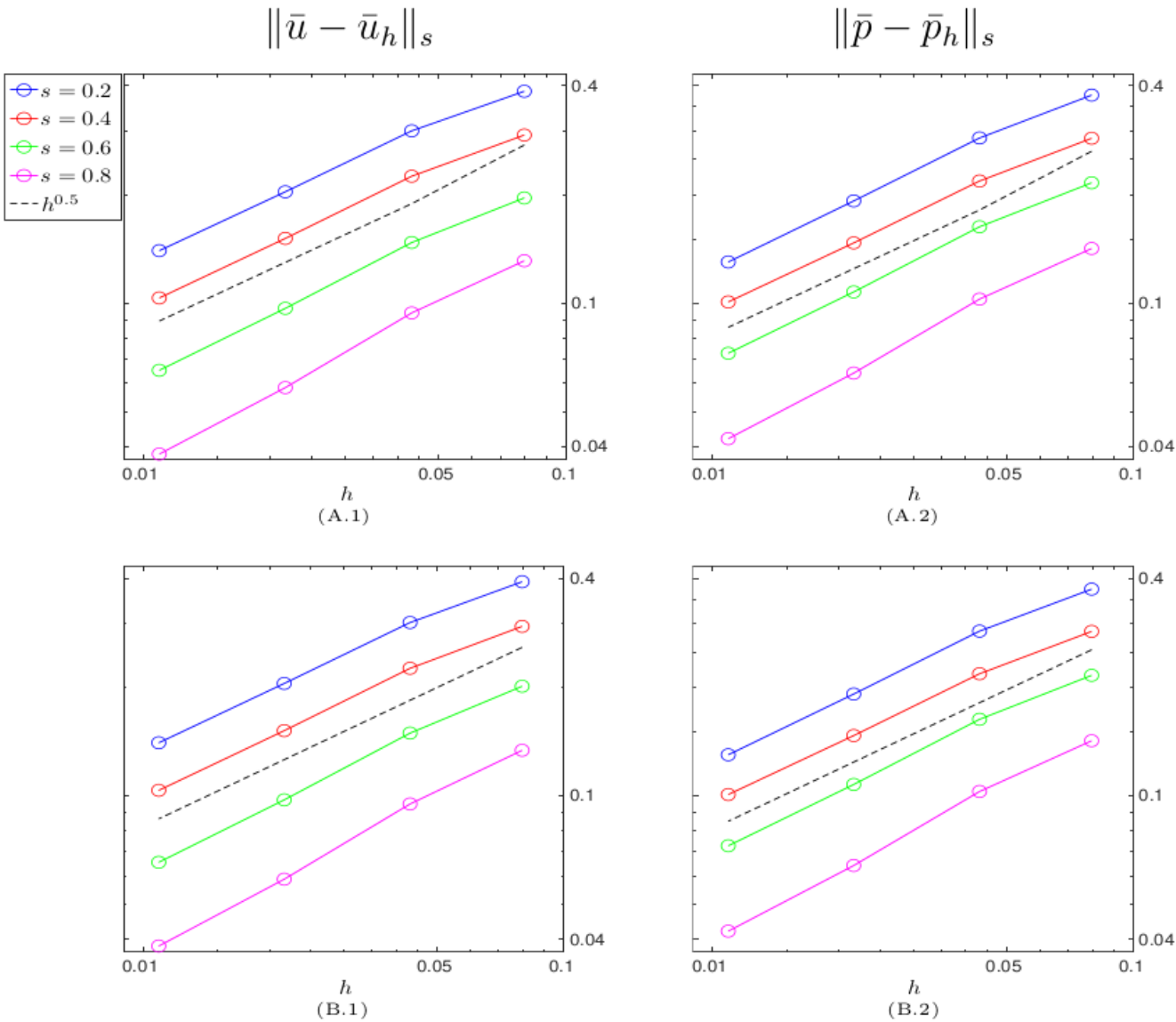} \\
	\caption{Experimental rates of convergence for $\| \bar{u} - \bar{u}_h \|_s$ and $\| \bar{p} - \bar{p}_h \|_s$ considering the fully discrete (A.1)--(A.2) and semidiscrete schemes (B.1)--(B.2) for $s \in \{0.2,0.4,0.6,0.8\}$.}
\label{fig:ex-1.1}
\end{figure}

%%%%%%%%%%%%%%%%%%%%%%%%%%%%%%%%%%%%%%%%
%%%%%%%%%%%%%%%%%%%%%%%%%%%%%%%%%%%%%%%%
%%%%%%%%                   %%%%%%%%%%%%%
%%%%%%%% EXAMPLE 1 - FIG.2 %%%%%%%%%%%%%
%%%%%%%%                   %%%%%%%%%%%%%
%%%%%%%%%%%%%%%%%%%%%%%%%%%%%%%%%%%%%%%%
%%%%%%%%%%%%%%%%%%%%%%%%%%%%%%%%%%%%%%%%

\begin{figure}[!ht]
\centering
		\includegraphics[trim={0 0 0 0},clip,width=12.51cm,height=10.71cm,scale=0.4]{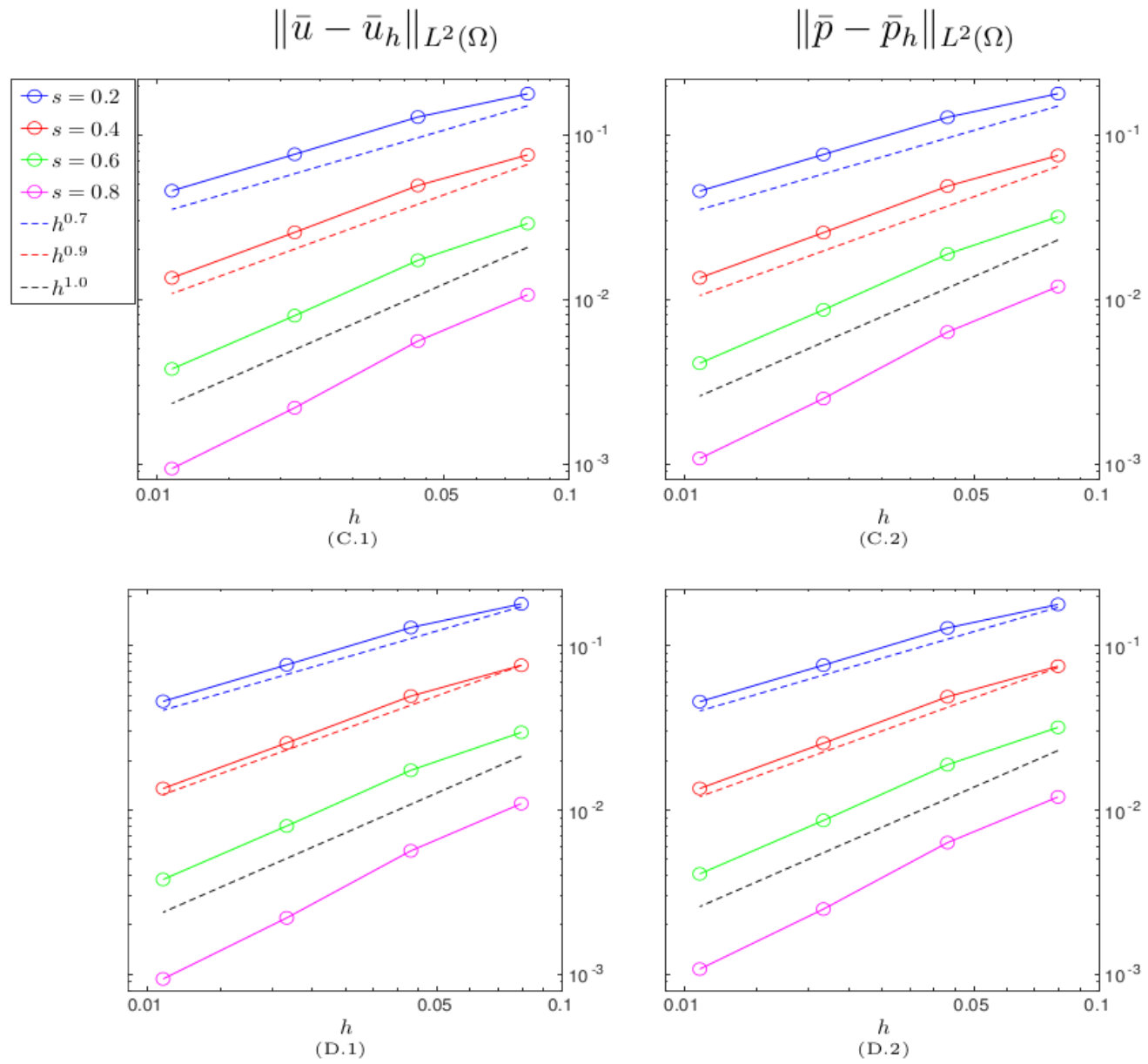} \\
\caption{Experimental rates of convergence for $\| \bar{u} - \bar{u}_h \|_{L^2(\Omega)}$ and $\| \bar{p} - \bar{p}_h \|_{L^2(\Omega)}$ considering the fully discrete (C.1)--(C.2) and semidiscrete schemes (D.1)--(D.2) for $s \in \{0.2,0.4,0.6,0.8\}$.}
\label{fig:ex-1.2}
\end{figure}

%%%%%%%%%%%%%%%%%%%%%%%%%%%%%%%%%%%%%%%%
%%%%%%%%%%%%%%%%%%%%%%%%%%%%%%%%%%%%%%%%
%%%%%%%%                   %%%%%%%%%%%%%
%%%%%%%% EXAMPLE 1 - FIG.3 %%%%%%%%%%%%%
%%%%%%%%                   %%%%%%%%%%%%%
%%%%%%%%%%%%%%%%%%%%%%%%%%%%%%%%%%%%%%%%
%%%%%%%%%%%%%%%%%%%%%%%%%%%%%%%%%%%%%%%%

\begin{figure}[!ht]
\centering
		\includegraphics[trim={0 0 0 0},clip,width=12.51cm,height=10.71cm,scale=0.4]{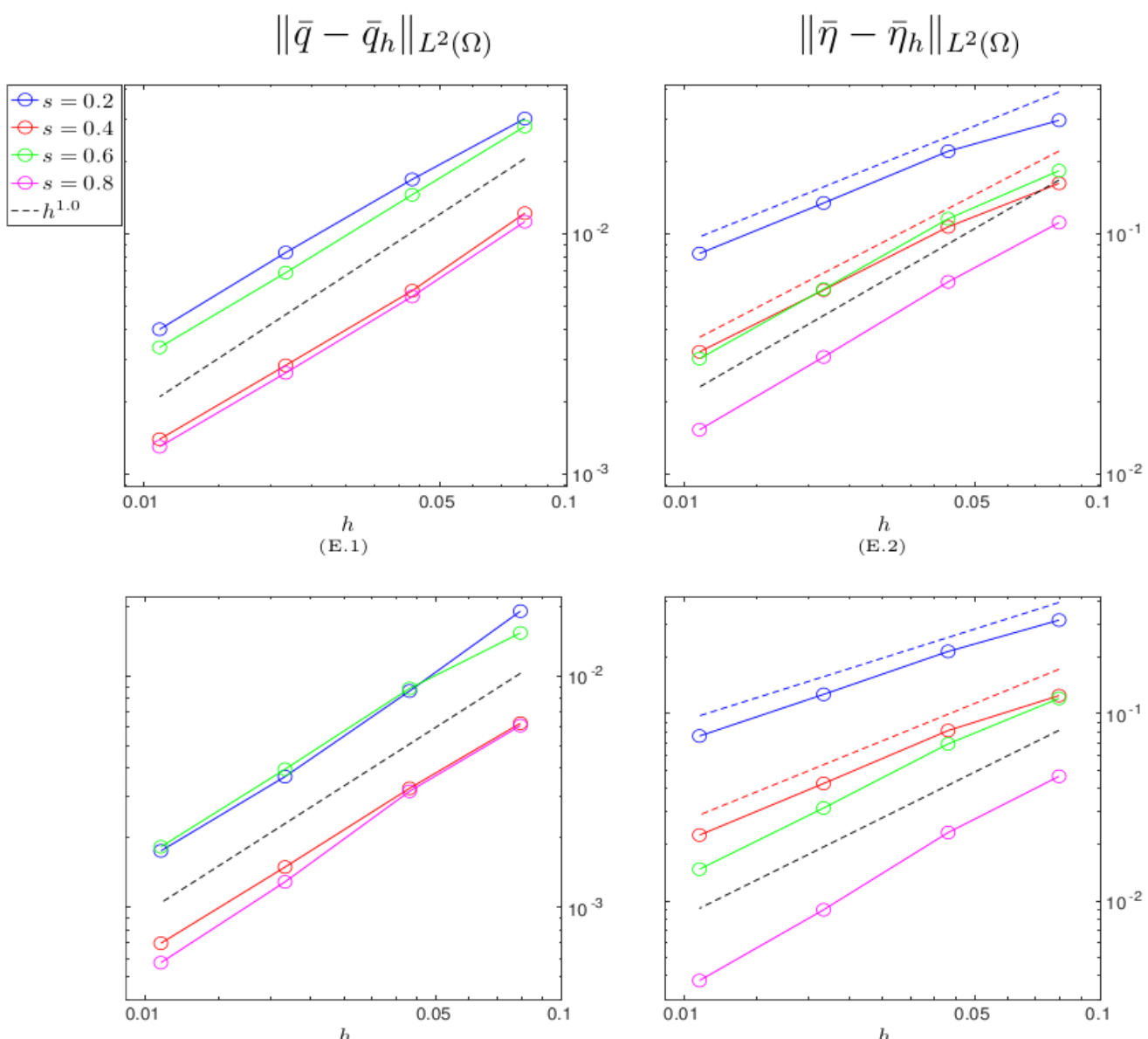} \\
\caption{Experimental rates of convergence for $\| \bar{q} - \bar{q}_h \|_{L^2(\Omega)}$ and $\| \bar{\eta} - \bar{\eta}_h \|_{L^2(\Omega)}$ considering the fully discrete (E.1)--(E.2) and semidiscrete schemes (F.1)--(F.2) for $s \in \{0.2,0.4,0.6,0.8\}$.}
\label{fig:ex-1.3}
\end{figure}

%%%%%%%%%%%%%%%%%%%%%%%%%%%%%%%%%%%%%%%%
%%%%%%%%%%%%%%%%%%%%%%%%%%%%%%%%%%%%%%%%
%%%%%%%%                   %%%%%%%%%%%%%
%%%%%%%% EXAMPLE 1 - FIG.4 %%%%%%%%%%%%%
%%%%%%%%                   %%%%%%%%%%%%%
%%%%%%%%%%%%%%%%%%%%%%%%%%%%%%%%%%%%%%%%
%%%%%%%%%%%%%%%%%%%%%%%%%%%%%%%%%%%%%%%%

\begin{figure}[!ht]
\centering
		\includegraphics[trim={0 0 0 0},clip,width=13.1cm,height=10.1cm,scale=0.4]{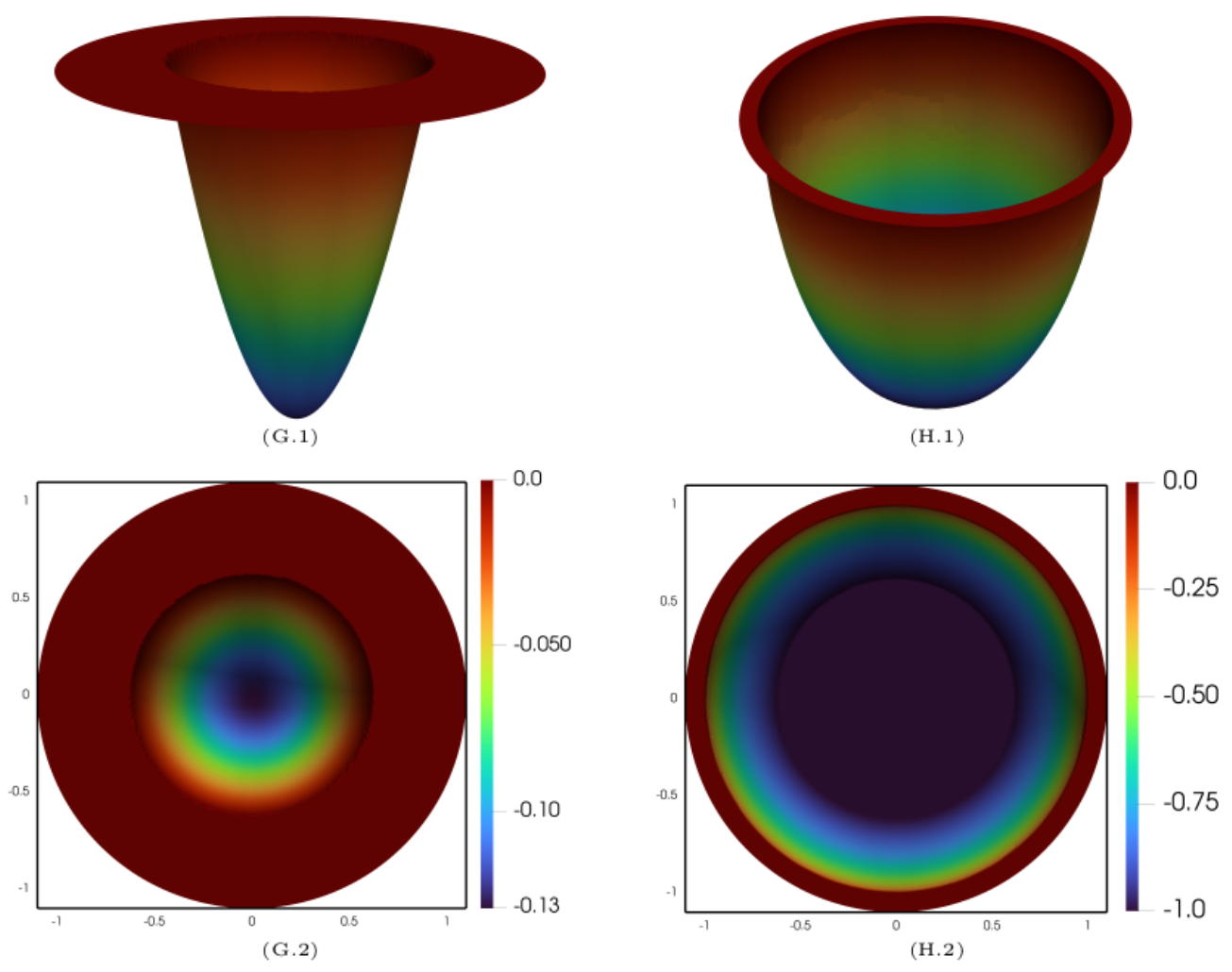} \\
\caption{Finite element solutions $\bar{\mathsf{q}}_h$ (left) and $\bar{\upeta}_h$ (right), obtained by the semidiscrete scheme with $s = 0.4$. The sparse behavior in the control variable $\bar{\mathsf{q}}_h$ is evident. In addition, a singular behavior can be observed for $\bar{\upeta}_h$ near the boundary $\partial \Omega$.}
\label{fig:ex-1.4}
\end{figure}

Figures \ref{fig:ex-1.1}, \ref{fig:ex-1.2}, and \ref{fig:ex-1.3} show the results for the fully discrete and semidiscrete schemes. Figure \ref{fig:ex-1.1} shows the experimental convergence rates for $\| \bar{u} - \bar{u}_h \|_s$ and $\| \bar{p} - \bar{p}_h \|_s$ for $s \in \{0.2, 0.4, 0.6, 0.8\}$. The experimental convergence rates for $\| \bar{u} - \bar{u}_h \|_{L^2(\Omega)}$ and $\| \bar{p} - \bar{p}_h \|_{L^2(\Omega)}$ are shown in Figure \ref{fig:ex-1.2}, while the results for $\| \bar{q} - \bar{q}_h \|_{L^2(\Omega)}$ and $\| \bar{\eta} - \bar{\eta}_h \|_{L^2(\Omega)}$ are shown in Figure \ref{fig:ex-1.3}. It can be observed that when $s\geq 0.5$ the experimental convergence rates for all involved approximation errors are in agreement with the error estimates obtained in Sections \ref{sec:FD_error_estimates} and \ref{sec:SD_scheme}. However, when $s < 0.5$, the reported experimental convergence rates exceed those predicted in our manuscript. We present a discussion in the following two Remarks. 
%\DQ{In Remarks \ref{rem:higher_rates} and \ref{rem:higher_rates2} we discuss this situation for the state and adjoint variables and for the control and subdifferential variables, respectively.}
%This case is discussed in \DQ{Remarks \ref{rem:higher_rates} and} \ref{rem:higher_rates2} }.

\begin{remark}[convergence rates: state and adjoint variables]\label{rem:higher_rates}
The Figures \ref{fig:ex-1.1} and \ref{fig:ex-1.2} show that the experimental convergence rates for the approximation errors of the state and adjoint variables exceed the rates predicted by the combination of Theorems \ref{thm:error_bounds}, \ref{thm:error_bound_control}, and \ref{thm:error_bound_control_SD} when $s < 0.5$, but agree in terms of the regularity
\begin{equation}
 H^{s + \frac{1}{2} -\epsilon}(\Omega), \qquad 0 < \epsilon < s + \tfrac{1}{2}.
\label{eq:maximal_regularity}
\end{equation}
The error bounds in Theorem \ref{thm:error_bounds} are based on the regularity results of Theorem \ref{thm:reg_prop_u_p}, which in turn are inspired by the estimates in Theorem \ref{thm:sobolev_reg} (\cite[Theorem 2.1]{MR4283703}). If $s \in (0,0.5]$, $\bar{u}, \bar{p} \in H^{2s - 2\epsilon}(\Omega)$ for every $\epsilon \in (0,s)$, which is weaker than \eqref{eq:maximal_regularity}. As explained in \cite[page 1921]{MR4283703}, one expects the solutions to be smoother than just $H^{2s}(\Omega)$ if the corresponding forcing term belongs to $H^r(\Omega)$, for some $r>0$; however, such a result of higher regularity cannot be derived from \cite[Theorem 2.1]{MR4283703}.
% Nevertheless, it is important to emphasize that the estimates in \cite[Theorem 2.1]{MR4283703} hold under the assumption that $\partial \Omega$ is \emph{merely} Lipschitz.
It is important to note that in our particular setting $\bar{u}$ and $\bar{p}$ defined in \eqref{eq:u_and_p} satisfy \eqref{eq:maximal_regularity}. Regarding the approximation in the $L^2(\Omega)$-norm, we note that it behaves as prescribed by the property of maximal regularity \eqref{eq:maximal_regularity}, but limited to the rate $\mathcal{O}(h)$; see \cite[\S 6.1]{MR4283703}.
\end{remark}

\begin{remark}[convergence rates: control and subgradient variables]\label{rem:higher_rates2}
Figure \ref{fig:ex-1.3} shows the experimental convergence rates for the approximation error of the control and subgradient variables for both the fully discrete and the semidiscrete schemes. As for the approximation of the control variable, the experimental convergence rate $\mathcal{O}(h)$ is observed for both methods for $s \in \{0.2, 0.4, 0.6, 0.8\}$, which exceeds the expected rate when $s < 1/2$. This behaviour is due to the sparse structure of our optimal control problem: the optimal control variable $\bar{q}$ is supported in a small region $\omega \subset \Omega$; see Figure \ref{fig:ex-1.4} for a graphical representation of this comment in a suitable finite element approximation. Consequently, the singularities
near the boundary that $\bar{q}$ could have as a consequence of the singular  behaviour of $\bar{u}$ and $\bar{p}$ are not present. This is in contrast to the behaviour of $\bar{\eta}$; see Figure \ref{fig:ex-1.4} for a graphical representation of a suitable finite element approximation. The latter also explains, in view of the comments in Remark \ref{rem:higher_rates}, the experimentally determined convergence rates obtained for the error within the approximation of $\bar{\eta}$.
% \alert{using the projection formula \eqref{eq:charac_eta}. In fact, using a standard rewriting of the projection formula in the form of maxima and minima, we have}
% \[
% \bar{p} - \mu \bar{\eta} = \bar{p} - \mu \Pi_{[-1,1]} \left( -\mu^{-1} \bar{p} \right) = \min \{0 , \bar{p} + \mu \} + \max \{0 , \bar{p} - \mu \}.
% \]
% In this specific example, where $\bar{p} \geq 0$, we find that $\bar{p} - \mu \bar{\eta} = \max\{0 , \bar{p} - \mu \}$. It is important to note that $\mu \neq 0$ implies that $\bar{p} - \mu\bar{\eta}$ is a constant function in a neighborhood of the boundary. Consequently, it does not exhibit the singularity of $\bar{p}$ near the edge. From Equation \eqref{eq:charac_q} can be deduced that the control $\bar{q}$ inherits this regularity.
% Notably, the obtained convergence rate aligns with the expected highest rate for and $\| \bar{u} - \bar{u}_{h} \|_{L^2(\Omega)}$, $\| \bar{u} - \bar{u}_{h} \|_{L^2(\Omega)}$, $\| \bar{p} - \bar{p}_{h} \|_{L^2(\Omega)}$, and (see \cite[\S 6.1]{MR4283703}).
\end{remark}

%
%%%%%%%%%%%%%%%%%%%%%%%%%%%%%%%%%%%%%%%%%%%%%%%%%%%%%%%%%%%%%
%%%%%%%%%%%%%%%%%%%%%%%%%%%%%%%%%%%%%%%%%%%%%%%%%%%%%%%%%%%%%
%%%%%%%%%%%%%%%%%%%%%%%%%%%%%%%%%%%%%%%%%%%%%%%%%%%%%%%%%%%%%
%%%%%%%%%%%%%%%%%%%%%%%%%%%%%%%%%%%%%%%%%%%%%%%%%%%%%%%%%%%%%
%%%%%%%%%%%%%%%%%%%%%%%%%%%%%%%%%%%%%%%%%%%%%%%%%%%%%%%%%%%%%
%%%%%%%%%%%%%%%%%%%%%%%%%%%%%%%%%%%%%%%%%%%%%%%%%%%%%%%%%%%%%
%%%%%%%%%%%%%%%%%%%%%%%%%%%%%%%%%%%%%%%%%%%%%%%%%%%%%%%%%%%%%
%%%%%%%%%%%%%%%%%%%%%%%%%%%%%%%%%%%%%%%%%%%%%%%%%%%%%%%%%%%%%

\bibliographystyle{siamplain}
\bibliography{sparse_ref}

\end{document}